\documentclass[12pt,reqno,a4paper]{amsart}

\usepackage[doi=false,isbn=false,style=alphabetic,sorting=nyt,backend=biber,maxnames=5,maxalphanames=5,giveninits=true]{biblatex}
\bibliography{tropicalVarietyProjection}

\usepackage[breaklinks,colorlinks,plainpages,hypertexnames=false,plainpages=false]{hyperref}
\hypersetup{urlcolor=blue, citecolor=blue, linkcolor=blue}

\usepackage[utf8]{inputenc}
\usepackage{amsmath}
\usepackage{amsthm}
\usepackage{amssymb}
\usepackage{amsfonts}
\usepackage{enumerate}
\usepackage{stmaryrd}
\usepackage{enumitem}
\usepackage{centernot}
\usepackage{todonotes}
\usepackage{mathrsfs}
\usepackage{listings}
\usepackage{multicol}
\usepackage[capitalise, noabbrev]{cleveref} %
\usepackage[noend]{algorithmic} %

\usepackage{mathdots}

\usepackage{array}
\newcolumntype{L}[1]{>{\raggedright\let\newline\\\arraybackslash\hspace{0pt}}m{#1}}
\newcolumntype{C}[1]{>{\centering\let\newline\\\arraybackslash\hspace{0pt}}m{#1}}
\newcolumntype{R}[1]{>{\raggedleft\let\newline\\\arraybackslash\hspace{0pt}}m{#1}}

\usepackage{tikz}
\usetikzlibrary{arrows}
\usetikzlibrary{decorations.pathreplacing}
\usetikzlibrary{matrix}
\usetikzlibrary{calc}
\usetikzlibrary{shapes}
\usetikzlibrary{patterns}
\usetikzlibrary{fit,backgrounds,scopes}
\newtheoremstyle{theoremstyle}
{10pt}      %
{5pt}       %
{\itshape}  %
{}          %
{\bfseries} %
{}         %
{ }      %
{}          %

\newtheoremstyle{algorithmstyle}
{10pt}      %
{5pt}       %
{}  %
{}          %
{\bfseries} %
{}         %
{ }      %
{}          %

\newtheoremstyle{examplestyle}
{10pt}      %
{5pt}       %
{}          %
{}          %
{\bfseries} %
{}         %
{ }      %
{}          %

\theoremstyle{theoremstyle}
\newtheorem{theorem}{Theorem}[section]
\newtheorem{lemma}[theorem]{Lemma}
\newtheorem{proposition}[theorem]{Proposition}

\theoremstyle{examplestyle}
\newtheorem{example}[theorem]{Example}
\newtheorem{definition}[theorem]{Definition}

\newtheorem{remark}[theorem]{Remark}
\newtheorem{convention}[theorem]{Convention}

\theoremstyle{algorithmstyle}
\newtheorem{algorithm}[theorem]{Algorithm}

\newcommand{\NN}{\mathbb{N}}

\newcommand{\RR}{\mathbb{R}}
\newcommand{\QQ}{\mathbb{Q}}
\newcommand{\ZZ}{\mathbb{Z}}
\newcommand{\FF}{\mathbb{F}}
\newcommand{\PP}{\mathbb{P}}

\newcommand{\suchthat}{\;\ifnum\currentgrouptype=16 \middle\fi|\;}
\newcommand{\bigmid}{\left.\vphantom{\Big\{} \suchthat \vphantom{\Big\}}\right.}

\DeclareMathOperator{\Trop}{Trop}

\DeclareMathOperator{\Res}{Res}

\newcommand\restr[2]{{\left.\kern-\nulldelimiterspace #1 \right|_{#2}}}

\newcommand{\overbar}[1]{\mkern 1.5mu\overline{\mkern-1.5mu#1\mkern-1.5mu}\mkern 1.5mu}

\begin{document}

\title[Computing zero-dimensional tropical varieties]{Computing zero-dimensional tropical varieties\\ via projections}
\author{Paul G\"orlach}
\address{Max Planck Institute for Mathematics in the Sciences\\
  Inselstra{\ss}e 22\\
  04103 Leipzig\\
  Germany
}
\email{paul.goerlach@mis.mpg.de}
\urladdr{https://personal-homepages.mis.mpg.de/goerlach}
\author{Yue Ren}
\address{Max Planck Institute for Mathematics in the Sciences\\
  Inselstra{\ss}e 22\\
  04103 Leipzig\\
  Germany
}
\email{yue.ren@mis.mpg.de}
\urladdr{https://yueren.de}

\author{Leon Zhang}
\address{University of California, Berkeley\\
  970 Evans Hall \#3840\\
  Berkeley, CA 94720-3840, USA
  USA
}
\email{leonyz@berkeley.edu}
\urladdr{https://math.berkeley.edu/~leonyz}

\subjclass[2010]{14T05, 13P10, 13P15, 68W30}

\date{\today}

\keywords{tropical geometry, tropical varieties, computer algebra.}

\begin{abstract}
  We present an algorithm for computing zero-dimensional tropical varieties using projections. Our main tools are fast unimodular transforms of lexicographical Gr\"obner bases. We prove that our algorithm requires only a polynomial number of arithmetic operations if given a Gr\"obner basis, and we demonstrate that our implementation compares favourably to other existing implementations. \newline Applying it to the computation of general positive-dimensional tropical varieties, we argue that the complexity for calculating tropical links is dominated by the complexity of the Gr\"obner walk.
\end{abstract}

\maketitle

\section{Introduction}

Tropical varieties are piecewise linear structures which arise from polynomial equations. They appear naturally in many areas of mathematics and beyond, such as geometry \cite{Mikhalkin05}, combinatorics \cite{AK06,FS05}, and optimisation \cite{ABGJ18}, as well as phylogenetics \cite{SS04,LMY18}, celestial mechanics \cite{HM06,HJ11}, and auction theory \cite{BK19,TY15}. Wherever they emerge, tropical varieties often provide a fresh insight into existing computational problems, which is why efficient algorithms and optimised implementations are of great importance.

Computing tropical varieties from polynomial ideals is a fundamentally important yet algorithmically challenging task, requiring sophisticated techniques from computational algebra and convex geometry. Currently, \textsc{Gfan} \cite{gfan} and \textsc{Singular} \cite{singular} are the only two programs capable of computing general tropical varieties. Both programs rely on a traversal of the Gr\"obner complex as initially suggested by Bogart, Jensen, Speyer, Sturmfels, and Thomas \cite{BJSST07}, and for both programs the initial bottleneck had been the computation of so-called tropical links. Experiments suggest that this bottleneck was resolved with the recent development of new algorithms \cite{Chan13,HR18}. However the new approaches still rely on computations that are known to be very hard, \cite{Chan13} on elimination and \cite{HR18} on root approximation to an unknown precision.

In this paper, we study the computation of zero-dimensional tropical varieties, which is the key computational ingredient in \cite{HR18}, but using projections, which is the key conceptual idea in \cite{Chan13}. We create a new algorithm for computing zero-dimensional tropical varieties that only requires a polynomial amount of field operations if we start with a Gr\"obner basis, and whose timings compare favourably with other implementations even if we do not. In particular, we argue that in the computation of general tropical varieties, the calculation of so-called tropical links becomes computationally insignificant compared to the Gr\"obner walk required to traverse the tropical variety.

\smallskip

Note that projections are a well-studied approach in polynomial systems solving, see \cite{Sturmfels02,DE05} for an overview on various techniques. Our approach can be regarded as a non-Archimedean analogue of that strategy, since tropical varieties can be regarded as zeroth-order approximation of the solutions in the topology induced by the valuation.

\smallskip

Our paper is organised as follows:
In Section~\ref{sec:unimodular}, we introduce a special class of unimodular transformations and study how they act on generic lexicographical Gr\"obner bases. In Section~\ref{sec:algorithm}, we explain our main algorithm for reconstructing zero-dimensional tropical varieties from their projections,
while Section~\ref{sec:implementation} touches upon some technical details of the implementation. In Section~\ref{sec:timings}, we compare the performance of our algorithm against the root approximation approach, while Section~\ref{sec:complexity} analyses the complexity of our algorithm.

Implementations of all our algorithms can be found in the \textsc{Singular} library
\texttt{tropicalProjection.lib}. Together with the data for the timings, it is available at \href{https://software.mis.mpg.de}{https://software.mis.mpg.de}, and will also be made publicly available as part of the official \textsc{Singular} distribution.

\subsection*{Acknowledgment} The authors would like to thank Avi Kulkarni (MPI Leipzig) for his \textsc{Magma} script for solving polynomial equations over $p$-adic numbers, Marta Panizzut (TU Berlin) and Bernd Sturmfels (MPI Leipzig + UC Berkeley) for the examples of tropical cubic surfaces with $27$ distinct lines, as well as Andreas Steenpa{\ss} (TU Kaiserslautern) for his work on the \textsc{Singular} library \texttt{modular.lib} \cite{modularlib}.

\section{Background}

For the sake of notation, we briefly recall some basic notions of tropical algebraic geometry and computational algebra that are of immediate relevance to us. In tropical geometry, our notation closely follows that of \cite{MS15}.

\begin{convention}
  For the remainder of the article, let $K$ be a field with non-trivial valuation $\nu\colon K^* \rightarrow\RR$ and fix a multivariate polynomial ring $K[\mathbf x] := K[x_1, \ldots, x_n]$ as well as a multivariate Laurent polynomial ring $K[\mathbf x^{\pm}] := K[x_1^{\pm}, \ldots, x_n^{\pm}]$.

  Moreover, given a Laurent polynomial ideal $I\subseteq K[\mathbf x^{\pm}]$, we call a finite subset $G\subseteq I$ a \emph{Gr\"obner basis} with respect to a monomial ordering $\prec$ on $K[\mathbf{x}]$ if $G$ consists of polynomials and forms a Gr\"obner basis of the polynomial ideal $I\cap K[\mathbf x]$ with respect to $\prec$ in the conventional sense, see for example \cite[\S 1.6]{GreuelPfister02}. All our Gr\"obner bases are reduced.

  Finally, a \emph{lexicographical Gr\"obner basis} will be a Gr\"obner basis with respect to the lexicographical ordering $\prec_{\rm lex}$ with $x_n \prec_{\rm lex} \dots \prec_{\rm lex} x_1$.
\end{convention}

For the purposes of this article, the following definition of tropical varieties in terms of coordinate-wise valuations of points in solution sets suffices.

\begin{definition}[Tropical variety]
  Let $I \subseteq K[\mathbf x^{\pm}]$ be a Laurent polynomial ideal. The \emph{tropical variety} $\Trop(I) \subseteq \RR^n$ is given by
  \[ \Trop(I):=\text{cl}\Big(\big\{(\widehat{\nu}^\text{al}(p_1),\ldots,\widehat{\nu}^\text{al}(p_n)) \in \RR^n \mid (p_1,\ldots,p_n) \in V_{{\widehat K}^{\text{al}}}(I)\big\}\Big),\]
  where ${\widehat K}^{\text{al}}$ denotes the algebraic closure of the completion of $K$, so that $\nu$ extends uniquely to a valuation $\widehat{\nu}^\text{al}$ on ${\widehat K}^{\text{al}}$, $V_{{\widehat K}^{\text{al}}}(I)$ denotes the affine variety of $I$ over ${\widehat K}^{\text{al}}$, and $\text{cl}(\cdot)$ is the closure in the euclidean topology.
\end{definition}

In this article, our focus lies on zero-dimensional ideals $I \subseteq K[\mathbf x^{\pm}]$, in which case $\Trop(I)$ is a finite set of $\deg(I)$ points if each point $w\in\Trop(I)$ is counted with the multiplicity corresponding to the number of solutions $p\in V_{{\widehat K}^{\text{al}}}(I)$ with ${\widehat \nu}^{\text{al}}(p)=w$. %

In the univariate case, the tropical variety of an ideal $I = (f) \subseteq K[x^{\pm}_1]$ simply consists of the negated slopes in the Newton polygon of $f$ \cite[Proposition II.6.3]{Neukirch}. Our approach for computing zero-dimensional tropical varieties of multivariate ideals is based on reducing computations to the univariate case.

\begin{definition}
 We say that a zero-dimensional ideal $I \subseteq K[\mathbf{x}^{\pm}]$ is in \emph{shape position} if the projection onto the last coordinate $p_n \colon (K^\ast)^n \to K^\ast$, $(a_1,\ldots,a_n) \mapsto a_n$ defines a closed embedding $\restr{p_n}{V(I)} \colon V(I) \hookrightarrow K^\ast$.
\end{definition}

In this article, we will concentrate on ideals that are in shape position. Lemma~\ref{lem:shapeFormGroebnerBasis} shows an easy criterion to decide whether a given ideal is in shape position, while Lemma~\ref{lem:shapeFormTransformation} shows how to coax degenerate ideals into shape position.

\begin{lemma}[{\cite[\S 4 Exercise 16]{CLO05}}] \label{lem:shapeFormGroebnerBasis}
  A zero-dimensional ideal $I \subseteq K[\mathbf{x}^{\pm}]$ is in shape position if and only if its (reduced) lexicographical Gröbner basis is of the form
  \begin{equation}\tag{SP}\label{eq:shapePosition}
    G = \{f_n,\, x_{n-1}-f_{n-1},\, \ldots,\, x_2-f_2,\, x_1-f_1\}
  \end{equation}
  for some univariate polynomials $f_1, \ldots, f_n \in K[x_n]$. The polynomials $f_i$ are unique.
\end{lemma}

\begin{lemma}\label{lem:shapeFormTransformation}
  Let $I\subseteq K[\mathbf{x}^{\pm}]$ be a zero-dimensional ideal. Then there exists a euclidean dense open subset $\mathcal V\subseteq \RR^{n-1}$ such that for any $(u_1,\dots,u_{n-1})\in\mathcal V\cap \ZZ^{n-1}$ the unimodular transformation
  \[ \Phi_u\colon\quad K[\mathbf{x}^{\pm}] \rightarrow K[\mathbf{x}^{\pm}],\quad x_i\mapsto
    \begin{cases}
      x_i &\text{if } i<n,\\
      x_n \prod_{i=1}^{n-1} x_i^{u_{i}} &\text{if } i=n
    \end{cases}\]
  maps $I$ into an ideal in shape position.
\end{lemma}

\begin{proof}
  Without loss of generality, we may assume that the field $K$ is algebraically closed. For any $u=(u_1,\dots,u_{n-1})\in\mathcal \ZZ^{n-1}$, let $f_u \colon (K^*)^n \to (K^*)^n$ be the torus automorphism induced by $\Phi_u$, so that $V(\Phi_u(I)) = f_u^{-1}(V(I))$. Then the transformed ideal $\Phi_u(I)$ is in shape position if and only if the map $p_n \circ f_u^{-1} \colon (K^*)^n \to K^*$, $(a_1,\ldots,a_n) \mapsto a_n\cdot \prod_{i=1}^{n-1} a_i^{-u_{i}}$ is injective on the finite set $V(I)$.

  For $a \in (K^*)^n \setminus \{(1,\dots,1)\}$, the set $N_a := \{w \in \ZZ^n \mid a_1^{w_1} \dots a_n^{w_n} = 1\}$ is a $\ZZ^n$-sub\-lat\-tice of positive corank. Hence, $W_a := \{w \in \RR^{n-1} \mid (w_1,\ldots,w_{n-1},-1) \in N_a \otimes_\ZZ \RR\}$ is a proper affine subspace of $\RR^{n-1}$. By definition, for any two elements $b\neq c\in V(I)$, we have
  \[ p_n\circ f_u^{-1}(b) = p_n\circ f_u^{-1}(c)\quad \Leftrightarrow\quad (u_{1},\dots, u_{n-1})\in W_{b^{-1}c}. \]
  Thus, $\Phi_u(I)$ is in shape position if and only if $(u_{1},\ldots,u_{n-1}) \in \mathcal V \cap \ZZ^{n-1}$, where $\mathcal V := \RR^{n-1} \setminus \bigcup_{b \neq c \in V(I)} W_{b^{-1}c}$ is a euclidean dense open subset of $\RR^{n-1}$.
\end{proof}

\section{Unimodular transformations on lexicographical Gröbner bases}\label{sec:unimodular}

In this section, we introduce a special class of unimodular transformations and describe how they operate on lexicographical Gr\"obner bases in shape position.

\begin{definition}
  We will consider unimodular transformations indexed by the set
  \[\mathcal U := \{u \in \ZZ^n \mid \exists\: 1 \leq \ell \lneq n: u_\ell = -1 \text{ and } u_i \geq 0 \text{ for all } i\neq \ell\}.\]
  For any $u \in \mathcal U$, we define a unimodular ring automorphism
  \[ \varphi_u \colon K[\mathbf{x}^{\pm}] \to K[\mathbf{x}^{\pm}], \qquad
      x_i \mapsto \begin{cases} x_1^{u_1} \cdots x_{\ell-1}^{u_{\ell-1}}x_\ell^1 x_{\ell+1}^{u_{\ell+1}}\cdots x_n^{u_n} & \text{if } i = \ell, \\ x_i & \text{otherwise}, \end{cases} \]
  and a linear projection
  \[\pi_u \colon \RR^n \twoheadrightarrow \RR, \qquad
    (w_1,\ldots,w_n) \mapsto -\sum_{i=1}^{n} u_i w_i.\]
    We call such a $\varphi_u$ a \emph{slim (unimodular) transformation concentrated at $\ell$}.
\end{definition}

While our slim unimodular transformations might seem overly restrictive, the next lemma states that they are sufficient to compute arbitrary projections of tropical varieties, which is what we will need in Section~\ref{sec:algorithm}.

\begin{lemma} \label{lem:skewProjectionViaElimination}
  Let $\varphi_u$ be a slim transformation concentrated at $\ell$.
  Then
  \[\pi_u(\Trop (I)) = \Trop(\varphi_u(I) \cap K[x_\ell^{\pm}]).\]
\end{lemma}
\begin{proof}
  We may assume that $K$ is algebraically closed. The ring automorphism $\varphi_u$ induces a torus automorphism $f_u \colon (K^\ast)^n \to(K^\ast)^n$ with $f_u^{-1}(V(I)) = V(\varphi_u(I))$, which in turn induces a linear transformation $h_u \colon \RR^n \xrightarrow{\cong} \RR^n$ mapping $\Trop (\varphi_u(I))$ to $\Trop(I)$:
  \begin{center}
    \begin{tikzpicture}
      \node (PLeft) {$K[\mathbf{x}^{\pm}]$};
      \node[anchor=base west,xshift=35mm] (PRight) at (PLeft.base east) {$K[\mathbf{x}^{\pm}]$};

      \node[anchor=base,yshift=-6mm] (TLeft) at ($(PLeft.base)!0.5!(PRight.base)$) {induces};

      \node[anchor=base,yshift=-20mm] (TLeft) at (PLeft.base) {$(K^\ast)^n$};
      \node[anchor=base,yshift=-20mm] (TRight) at (PRight.base) {$(K^\ast)^n$};

      \node[anchor=base,yshift=-20mm] (RLeft) at (TLeft.base) {$\RR^n$};
      \node[anchor=base,yshift=-20mm] (RRight) at (TRight.base) {$\RR^n$};

      \draw[<-] (PLeft) -- node[above] {$\varphi_u$} (PRight);
      \draw[->] (TLeft) -- node[below] {$f_u$} (TRight);
      \draw[->] (RLeft) -- node[above] {$h_u$} (RRight);

      \draw[->] ($(TLeft.south)+(-0.2,0)$) -- node[left] {$\nu$} ($(RLeft.north)+(-0.2,0)$);
      \draw[->] ($(TRight.south)+(-0.2,0)$) -- node[right] {$\nu$} ($(RRight.north)+(-0.2,0)$);
      \node[anchor=base,font=\footnotesize,yshift=2mm] (xLeft) at (PLeft.north) %
      {$x_\ell\prod_{i\neq\ell}x_i^{u_i}$};
      \node[anchor=base,font=\footnotesize,yshift=2mm] (xRight) at (PRight.north)
      {$x_\ell$}; %
      \node[anchor=south,font=\footnotesize] (zLeft) at (TLeft.north) {$(z_1,\ldots,z_n)$};
      \node[anchor=south,font=\footnotesize] (zRight) at (TRight.north) {$(z_1,\ldots,z_\ell\cdot\prod_{i\neq\ell} z_i^{u_i},\ldots,z_n)$};
      \node[anchor=north,font=\footnotesize] (wLeft) at (RLeft.south) {$(w_1,\ldots,w_n)$};
      \node[anchor=north,font=\footnotesize] (wRight) at (RRight.south) {$(w_1,\ldots,w_\ell+\sum_{i\neq\ell} u_i w_i,\ldots,w_n)$};

      \draw[<-|] (xLeft) -- (xRight);
      \draw[|->] (zLeft) -- (zRight);
      \draw[|->] (wLeft) -- (wRight);
    \end{tikzpicture}
  \end{center}
  Hence, with $p_{\ell} \colon \RR^n \twoheadrightarrow \RR$ denoting the projection onto the $\ell$-th coordinate:
    \[\Trop(\varphi_u(I) \cap K[x_\ell^{\pm}]) = p_{\ell}(\Trop( \varphi_u(I))) = (p_{\ell} \circ h_u^{-1})(\Trop (I)) = \pi_u(\Trop (I)). \qedhere\]
\end{proof}

The following easy properties of slim unimodular transformations serve as a basic motivation for their inception. They map polynomials to polynomials, which is important when working with software which only supports polynomial data. Moreover, they preserve saturation and shape position for zero-dimensional ideals, which is valuable as saturating and restoring shape position as in Lemma~\ref{lem:shapeFormTransformation} are two expensive operations.

\begin{lemma} \label{lem:unimodularShapeForm}
  For any slim transformation $\varphi_u$ and any zero-dimensional ideal $I \subseteq K[\mathbf{x}^{\pm}]$, we have
  \begin{enumerate}
  \item $\varphi_u(K[\mathbf x])\subseteq K[\mathbf x]$, \label{item:preservePolynomials}
  \item $\varphi_u(I) \cap K[\mathbf x] = \varphi_u(I \cap K[\mathbf x])$, \label{item:preserveSaturation}
  \item $I$ in shape position $\Leftrightarrow \varphi_u(I)$ in shape position. \label{item:preserveShapePosition}
  \end{enumerate}
\end{lemma}

\begin{proof}\
  From the definition, it is clear that polynomials get mapped to polynomials under $\varphi_u$, showing (\ref{item:preservePolynomials}).
  In particular, $\restr{\varphi_u}{K[\mathbf{x}]} \colon K[\mathbf{x}] \to K[\mathbf{x}]$ induces a morphism $\hat{f_u} \colon K^n \to K^n$ which on the torus $(K^*)^n$ restricts to an automorphism $f_u$ with $f_u^{-1}(V(I)) = V(\varphi_u(I))$. To show (\ref{item:preserveSaturation}), we need to see that $\hat{f}_u^{-1}(V(I \cap K[\mathbf{x}])) \subseteq K^n$ does not have irreducible components supported outside the torus $(K^*)^n$. Now, $V(I \cap K[\mathbf{x}])$ is the closure of $V(I) \subseteq (K^*)^n$ in $K^n$, so by zero-dimensionality of $I$, we have $V(I \cap K[\mathbf{x}]) \subseteq (K^*)^n$. Since $\hat{f}_u^{-1}((K^*)^n) = (K^*)^n$, this proves (\ref{item:preserveSaturation}).
  Finally, we note that $\varphi_u(x_n)=x_n$, so we have $p_n \circ f_u = p_n$, where $p_n \colon (K^*)^n \to K^*$ denotes the projection onto the last coordinate. Hence, $\restr{p_n}{V(I)}$ is a closed embedding if and only if $\restr{p_n}{V(\varphi_u(I))}$ is, proving (\ref{item:preserveShapePosition}).
\end{proof}

The next Algorithm~\ref{algo:groebnerTransformation} allows us to efficiently transform a lexicographical Gr\"obner basis of $I$ into a lexicographical Gr\"obner basis of $\varphi_u(I)$. This is the main advantage of slim unimodular transformations, which we will leverage to compute $\pi_u(\Trop (I))$.

\begin{algorithm}[Slim unimodular transformation of a Gröbner basis]\label{algo:groebnerTransformation}\
  \begin{algorithmic}[1]
    \REQUIRE{$(\varphi_u,G)$, where
      \begin{itemize}[leftmargin=*]
      \item $\varphi_u$ is a slim transformation concentrated at $\ell$,
      \item $G = \{f_n,\, x_{n-1}-f_{n-1},\,\ldots,\,x_2-f_2,\,x_1-f_1\}$ is the lexicographical Gr\"obner basis of an ideal $I\subseteq K[\mathbf{x}^{\pm}]$ in shape position as in \eqref{eq:shapePosition}.
      \end{itemize}}
    \ENSURE{$G'$, the lexicographical Gröbner basis of $\varphi_u(I)$.}
    \STATE \label{algstep:minimalPolynomial} In the univariate polynomial ring $K[x_n]$, compute $f_\ell'$ with $\deg(f_\ell')<\deg(f_n)$ and
    \[ f_\ell' \equiv \Big(x_n^{u_n}\cdot\prod_{\substack{i=1\\ i\neq \ell}}^{n-1} f_i^{u_i}\Big)^{-1} \cdot f_\ell \pmod{f_n}. \]
    \RETURN{$G' := \{f_n,\, x_{n-1}-f_{n-1},\,\ldots,\, x_\ell-f_\ell',\,\ldots,\,x_1-f_1\}$}.
  \end{algorithmic}
\end{algorithm}

\begin{proof}[Correctness of \cref{algo:groebnerTransformation}]
  The polynomial ideal $I \cap K[\mathbf{x}]$ is saturated with respect to the product of variables $x_1 \cdots x_n$, and is by assumption generated by $G$.
  This implies that $f_n$ is relatively prime to each $f_i$ for $i < n$ and to $x_n$. In particular, the inverse in $K[x_n]/(f_n)$ showing up in the definition of $f'_\ell$ is well-defined.
  The ideal $\varphi_u(I) \subseteq K[\mathbf{x}^{\pm}]$ is generated by
    \[\varphi_u(G) = \{f_n,\, x_{n-1}-f_{n-1},\,\ldots,\, \big(\prod_{\substack{i=1 \\ i \neq \ell}}^n x_i^{u_i}\big) x_\ell-f_\ell,\,\ldots,\,x_1-f_1\}.\]
  Note that the expression $(\prod_{i \neq \ell} x_i^{u_i}) x_\ell-f_\ell$ is equivalent to $x_\ell-f_\ell'$ modulo the ideal $(f_n,\, x_i - f_i \mid i \neq \ell, n)$. It follows that $\varphi_u(I)$ is generated by $G'$, and it is clear that $G'$ is a lexicographical Gr\"obner basis.
\end{proof}

\section{Computing zero-dimensional tropical varieties via projections}\label{sec:algorithm}

In this section, we assemble our algorithm for computing $\Trop(I)$ from a zero-dimensional ideal $I \subseteq K[\mathbf x^{\pm}]$. This is done in two stages, see Figure~\ref{fig:algorithm}: In the first stage, we project $\Trop(I)$ onto all coordinate axes of $\RR^n$. In the second stage, we iteratively glue the coordinate projections together until $\Trop(I)$ is fully assembled.

\begin{figure}[h]
  \centering
  \begin{tikzpicture}[y=7.5mm,every node/.style={font=\footnotesize}]
    \draw (-3,0) -- (4,0) node[right] (lin1) {$\RR\cdot e_1$};
    \node[anchor=base west,xshift=-2mm] at (lin1.base east) {$\supseteq p_{\{1\}}(\Trop(I))$};
    \draw (0,-3) -- (0,4) node[above] (lin2) {$\RR\cdot e_2$};
    \node[anchor=base west,xshift=-2mm] at (lin2.base east) {$\supseteq p_{\{2\}}(\Trop(I))$};

    \draw[dashed,opacity=0.5]
    (1,-3) -- (1,4)
    (2,-3) -- (2,4)
    (3,-3) -- (3,4)
    (-1,-3) -- (-1,4)
    (-2,-3) -- (-2,4)
    (-3,1) -- (4,1)
    (-3,2) -- (4,2)
    (-3,3) -- (4,3)
    (-3,-2) -- (4,-2)
    (-3,-1) -- (4,-1);

    \fill[draw,fill=white] (1,0) circle (1.75pt)
    (2,0) circle (1.75pt)
    (3,0) circle (1.75pt)
    (-2,0) circle (1.75pt)
    (-1,0) circle (1.75pt)
    (0,1) circle (1.75pt)
    (0,2) circle (1.75pt)
    (0,3) circle (1.75pt)
    (0,-2) circle (1.75pt)
    (0,-1) circle (1.75pt);

    \draw[blue] (2,-3) -- (-2.667,4) node[above] (linu) {$\RR\cdot u$};
    \node[blue,anchor=base east,xshift=2mm] at (linu.base west) {$\pi_u(\Trop(I))\subseteq$};
    \draw[blue,dashed,opacity=0.7]
    (-1,-1) -- ++(40:12.5mm) coordinate (p1) -- ++(40:42.5mm)
    (-1,-1) -- ++(220:22.5mm)
    (-2,2) -- ++(40:5mm) coordinate (p2) -- ++(40:15mm)
    (-2,2) -- ++(220:12.5mm)
    (1,3) -- ++(220:22.5mm) coordinate (p3) -- ++(220:29mm)
    (1,3) -- ++(40:8mm)
    (2,-2) -- ++(220:5mm) coordinate (p4) -- ++(220:5mm)
    (2,-2) -- ++(40:25mm)
    (3,1) -- ++(220:27.5mm) coordinate (p5) -- ++(220:17.5mm)
    (3,1) -- ++(40:12.5mm);
    \fill[draw=blue,fill=white]
    (p1) circle (1.75pt)
    (p2) circle (1.75pt)
    (p3) circle (1.75pt)
    (p4) circle (1.75pt)
    (p5) circle (1.75pt);
    \fill (1,3) circle (2pt)
    (3,1) circle (2pt)
    (-2,2) circle (2pt)
    (-1,-1) circle (2pt)
    (2,-2) circle (2pt);
  \end{tikzpicture}\vspace{-2mm}
  \caption{Computing zero-dimensional tropical varieties via projections.}
  \label{fig:algorithm}
\end{figure}
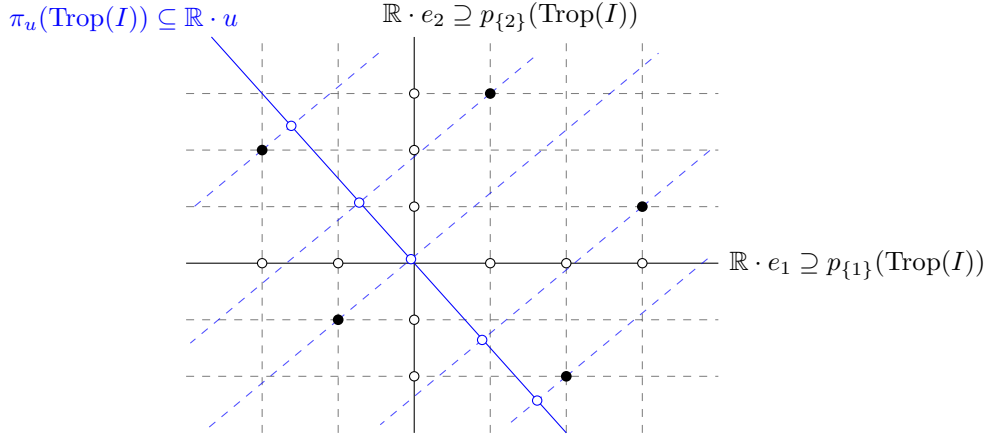

For the sake of simplicity, all algorithms contain some elements of ambiguity to minimise the level of technical detail. To see how these ambiguities are resolved in the actual implementation, see Section~\ref{sec:implementation}. Moreover, we will only consider $\Trop(I)$ as points in $\RR^n$ without multiplicities. It is straightforward to generalise the algorithms to work with $\Trop(I)$ as points in $\RR^n$ with multiplicities, which is how they are implemented in \textsc{Singular}.

The following algorithm merges several small projections into a single large projection. For clarity, given a finite subset $A\subseteq\{1,\dots,n\}$, we use $\RR^A$ to denote the linear subspace of $\RR^n$ spanned by the unit vectors indexed by $A$ and $p_A$ to denote the projection $\RR^n\twoheadrightarrow \RR^A$.

\begin{algorithm}[gluing projections]\label{algo:glue}\
  \begin{algorithmic}[1]
    \REQUIRE{$(G,p_{A_1}(\Trop(I)),\dots,p_{A_k}(\Trop(I)))$, where
      \begin{itemize}[leftmargin=*]
      \item $G$ is the lexicographical Gr\"obner basis of a zero-dimensional ideal ${I\subseteq K[\mathbf x^{\pm}]}$ in shape position as in \eqref{eq:shapePosition},
      \item $A_1,\ldots,A_k\subseteq \{1,\ldots,n\}$ are non-empty sets.
      \end{itemize}}
    \ENSURE{$p_A(\Trop(I)) \subseteq \RR^A$, where $A := A_1 \cup \ldots \cup A_k$.}
    \STATE Construct the candidate set
    \[ T := \Big\{w \in \RR^A \bigmid p_{A_i}(w) \in p_{A_i}(\Trop(I)) \text{ for } i=1,\ldots,k\Big\}.\]\vspace{-1em}
    \STATE \label{algstep:pickU} Pick a slim transformation $\varphi_u$ such that the following map is injective:
    \[ \restr{\pi_u}{T} \colon T \to \RR, \qquad (w_i)_{i\in A}\mapsto -\sum_{i\in A} u_i w_i.\]\vspace{-1em}
    \STATE \label{algstep:callAlgoUnimodular} Using \cref{algo:groebnerTransformation}, transform $G$ into a Gr\"obner basis $G'$ of $\varphi_{u}(I)$:
    \[ G' := \{f_n,x_{n-1}-f_{n-1},\ldots,x_\ell-f_\ell',\ldots,x_1-f_1\}. \]\vspace{-1em}
    \STATE \label{algstep:computeMinPoly} Compute the minimal polynomial $\mu \in K[z]$ of $\overbar{f_\ell'} \in K[x_n]/(f_n)$ over $K$ and
    read off $\Trop(\mu) \subseteq \RR$ from its Newton polygon.
    \RETURN{$\{w \in T \mid \pi_u(w) \in \Trop(\mu)\}$.}
  \end{algorithmic}
\end{algorithm}

\begin{proof}[Correctness of \cref{algo:glue}]
  First, we argue that line~\ref{algstep:pickU} can be realised, i.e., we show the existence of a slim unimodular transformation $\varphi_u$ such that $\pi_u$ is injective on the candidate set $T$. Pick $\ell \neq n$ and denote $B := \{1,\ldots,n\} \setminus \{\ell\}$. It suffices to show that the set
    \[Z := \{v \in \RR_{\geq 0}^B \mid \restr{\pi_{v-e_\ell}}{T} \text{ is injective}\} \subseteq \RR^B \]
  contains an integer point. By the definition of $\pi_{v-e_\ell}$, we see that
  \[Z = \RR_{\geq 0}^B \setminus \bigcup_{w \neq w' \in T} H_{w-w'}, \text{ where } H_{w-w'} := \Big\{v \in \RR^B \bigmid \sum_{i \in B} (w_i-w_i') v_i = z_\ell\Big\}. \]
  This describes $Z$ as the complement of an affine hyperplane arrangement in $\RR^B$ inside the positive orthant. Therefore, $Z$ must contain an integer point.

  Next, we note that the candidate set $T$ contains $p_A(\Trop(I))$ by construction, so injectivity of $\restr{\pi_u}{T}$ shows that $p_A(\Trop(I)) = \{w \in T \mid \pi_u(w) \in \pi_u(\Trop(I))\}$. Therefore, the correctness of the output will follow from showing $\pi_u(\Trop(I)) = \Trop(\mu)$. By \cref{lem:skewProjectionViaElimination}, it suffices to prove that $\mu(x_\ell) \in K[x_\ell]$ generates the elimination ideal $\varphi_u(I) \cap K[x_\ell^{\pm}]$.

  For this, we observe that reducing a univariate polynomial $g \in K[x_\ell]$ with respect to the lexicographical Gröbner basis $G'$ substitutes $x_\ell$ by $f_\ell'$ to obtain a univariate polynomial in $K[x_n]$ and then reduces the result modulo $f_n$. In particular, this shows that such $g \in K[x_\ell]$ lies in the ideal $\varphi_u(I)$ if and only if $g(\overbar{f_\ell'}) = 0$ in $K[x_n]/(f_n)$. Hence, the elimination ideal $\varphi_u(I) \cap K[x_\ell]$ is generated by $\mu(x_\ell)$.
\end{proof}

The next algorithm computes $\Trop(I)$ by projecting it onto all coordinate axes and gluing the projections together via Algorithm~\ref{algo:glue}.

\begin{algorithm}[tropical variety via projections]\label{algo:main}\
  \begin{algorithmic}[1]
    \REQUIRE{$G = \{f_n,\, x_{n-1}-f_{n-1},\, \ldots,\, x_2-f_2,\, x_1-f_1\}$, the lexicographical Gr\"obner basis of a zero-dimensional ideal $I\subseteq K[\mathbf x^{\pm}]$ in shape position as in \eqref{eq:shapePosition}.}
    \ENSURE{$\Trop(I) \subseteq \RR^n$}
    \STATE Compute the projection onto the last coordinate $p_{\{n\}}(\Trop(I)) = \Trop(f_n)$. \label{algstep:projectPrincipal1}
    \FOR{$k\in\{1,\ldots,n-1\}$}
    \STATE Compute the minimal polynomial $\mu_k \in K[z]$ of $\overbar{f_k} \in K[x_n]/(f_n)$ over $K$ and read off the projection
      $p_{\{k\}}(\Trop(I)) = \Trop(\mu_k)$.%
       \label{algstep:projectPrincipal2}
    \ENDFOR
    \STATE Initialise a set of computed projections $W\!:=\!\{p_{\{1\}}(\Trop(I)),\dots,p_{\{n\}}(\Trop(I))\}$. \hspace{-1cm}
    \WHILE{$W \not\ni p_{\{1,\ldots,n\}}(\Trop(I))$}
    \STATE \label{algstep:pickingGlue} Pick projections $p_{A_1}(\Trop(I)),\dots,p_{A_k}(\Trop(I))\in W$ to be merged such that $p_A(\Trop(I)) \notin W$ for $A := A_1\cup \dots \cup A_k$.
    \STATE Using Algorithm~\ref{algo:glue}, compute $p_A(\Trop(I))$. \label{algstep:pickMerge}
    \STATE $W:=W\cup\{p_A (\Trop(I))\}$.
    \ENDWHILE
    \RETURN{$p_{\{1,\ldots,n\}}(\Trop(I))$.}
  \end{algorithmic}
\end{algorithm}

\begin{proof}[Correctness of \cref{algo:main}]
  Since $G$ is the lexicographical Gröbner basis of $I$, the elimination ideal $I \cap K[x_n^{\pm}]$ is generated by $f_n$, so we indeed have the equality $p_{\{n\}}(\Trop(I)) = \Trop(f_n)$ in line~\ref{algstep:projectPrincipal1}. The equality $p_{\{k\}}(\Trop(I)) = \Trop(\mu_k)$ in line~\ref{algstep:projectPrincipal2} holds because $\mu_k(x_\ell) \in K[x_\ell]$ generates the elimination ideal $I \cap K[x_\ell^{\pm}]$ by the same argument as in the proof of correctness of \cref{algo:glue}.

  In every iteration of the \texttt{while} loop, the set $W$ grows in size. Since there are only finitely many coordinate sets $A \subseteq \{1,\ldots, n\}$, we will after finitely many iterations compute $\Trop(I) = p_{\{1,\ldots,n\}}(\Trop(I))$, hence the while loop terminates.
\end{proof}

\begin{example}\label{ex:mainAlgorithm}
  Consider $K=\mathbb Q$ equipped with the 2-adic valuation and the ideal
  \[ I = ( \underbrace{2+x_3+x_3^2+x_3^3+2x_3^4}_{=:f_3}, x_2-\underbrace{2x_3}_{=:f_2}, x_1-\underbrace{4x_3}_{=:f_1} )\subseteq K[x_1^{\pm},x_2^{\pm},x_3^{\pm}]. \]
  This ideal is in shape position by \cref{lem:shapeFormGroebnerBasis}. From the Newton polygon of $f_3$, see \cref{fig:newtonPolygon} (left), it is not hard to see that
  \begin{align*}
    & p_{\{3\}}(\Trop(I)) = \Trop(f_3) = \{ -1,\mathbf{0},1 \}, \\
    & p_{\{2\}} (\Trop(I)) = \{\lambda + 1\mid \lambda\in p_{\{3\}}(\Trop(I))\} = \{ 0,\mathbf{1},2 \}, \\
    & p_{\{1\}} (\Trop(I)) = \{\lambda + 2\mid \lambda\in p_{\{3\}}(\Trop(I))\} =  \{ 1,\mathbf{2},3 \},
  \end{align*}
  where points with multiplicity $2$ are highlighted in bold.
  To merge $p_{\{1\}}(\Trop(I))$ and $p_{\{2\}}(\Trop(I))$, we consider the following projection that is injective on the candidate set $T:=p_{\{1\}}(\Trop(I))\times p_{\{2\}}(\Trop(I))$:
  \[ \pi_{(-1,3,0)}:\quad T\longrightarrow \RR, \quad (w_1,w_2)\longmapsto w_1-3w_2. \]
  The corresponding unimodular transformation $\varphi_{(-1,3,0)}$ sends $x_1$ to $x_1x_2^3$ and hence $\varphi_{(-1,3,0)}(I)$ is generated by $\{f_3,x_2-f_2, x_1x_2^3-4x_3\}$, which Algorithm~\ref{algo:groebnerTransformation} transforms into the following lexicographical Gr\"obner basis:
  \[ \varphi_{(-1,3)}(I) = \Big(f_3,\, x_2-f_2,\,  x_1-(\underbrace{\textstyle\frac{1}{4}x_3^3-\textstyle\frac{3}{8}x_3^2-\textstyle\frac{1}{8}x_3-\textstyle\frac{1}{8}}_{=:f_1'}) \Big). \]
  The minimal polynomial of $\overbar{f_1'}$ in $K[x_3]/(f_3)$ over $K$ can be computed as the resultant
  \[ \Res_{x_3}(f_3,x_1-f_1') = 8x_1^4+3x_1^3+\textstyle\frac{7}{2}x_1^2+\textstyle\frac{3}{4}x_1+\textstyle\frac{1}{2}. \]
  Figure~\ref{fig:newtonPolygon} (middle) shows the Newton polygon of the resultant, from which we see:
  \[ \Trop(\Res_{x_3}(f_3,x_1-f_1')) = \{ -3, \mathbf{-1}, 1 \}. \]
  Thus,
  \[ p_{\{1,2\}}(\Trop(I)) = \{ (3,2), \mathbf{(2,1)}, (1,0) \}. \]

  To merge $p_{\{1,2\}}(\Trop(I))$ and $p_{\{3\}}(\Trop(I))$, we consider the following projection that is injective on the candidate set $T:=p_{\{1,2\}}(\Trop(I))\times p_{\{3\}}(\Trop(I))$:
  \[ \pi_{(-1,0,3)}:\quad T\longrightarrow \RR, \quad (w_1,w_2,w_3)\longmapsto w_1-3w_3. \]
  The corresponding unimodular transformation $\varphi_{(-1,0,3)}$ sends $x_1$ to $x_1x_3^3$ and hence $\varphi_{(1,0,3)}(I)$ is generated by $\{f_3,x_2-f_2, x_1x_3^3-4x_3\}$, which Algorithm~\ref{algo:groebnerTransformation} transforms into the following lexicographical Gr\"obner basis:
  \[ \varphi_{(-1,0,3)}(I) = \Big( f_3,\, x_2-f_2,\, x_1-(\underbrace{2x_3^3-3x_3^2-x_3-1}_{=:f_1''}) \Big). \]
  Another resultant computation yields the minimal polynomial of $\overbar{f_1''} \in K[x_3]/(f_3)$ over $K$:
  \[ \Res_{x_3}(f_3,x_1-f_1'') = 8x_1^4+24x_1^3+224x_1^2+384x_1+2048. \]
  Figure~\ref{fig:newtonPolygon} (right) shows the Newton polygon of the resultant, from which we see:
  \[ \Trop(\Res_{x_3}(f_3,x_1-f_1'')) = \{ 0, \mathbf{2}, 4 \}, \]
  and thus
  \[ \Trop(I) = p_{\{1,2,3\}}(\Trop(I)) = \{ (3,2,1), \mathbf{(2,1,0)}, (1,0,-1) \}. \]
\end{example}

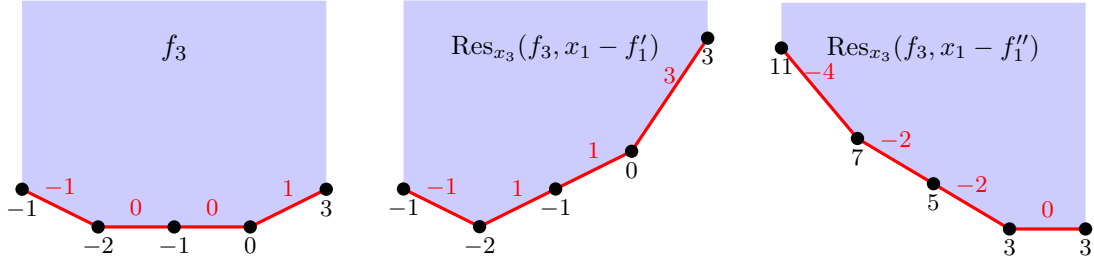
\begin{figure}[h]
  \centering
  \hfill
  \begin{tikzpicture}[y=5mm,every node/.style={font=\scriptsize}]
    \node (x0) at (0,1) {};
    \node (x1) at (1,0) {};
    \node (x2) at (2,0) {};
    \node (x3) at (3,0) {};
    \node (x4) at (4,1) {};

    \fill[blue!20] (x0.center) -- (x1.center) -- (x3.center) -- (x4.center) -- ++(0,5) -- ++(-4,0) -- cycle;
    \draw[red,very thick] (x0.center) -- node[red,above] {$-1$} (x1.center) -- node[red,above] {$0$} (x2.center) -- node[red,above] {$0$} (x3.center) -- node[red,above] {$1$} (x4.center);
    \fill (x0) circle (2.5pt);
    \fill (x1) circle (2.5pt);
    \fill (x2) circle (2.5pt);
    \fill (x3) circle (2.5pt);
    \fill (x4) circle (2.5pt);

    \node[anchor=north] at (x0) {$-1$};
    \node[anchor=north] at (x1) {$-2$};
    \node[anchor=north] at (x2) {$-1$};
    \node[anchor=north] at (x3) {$0$};
    \node[anchor=north] at (x4) {$3$};

    \node[font=\small] at (2,4.75) {$f_3$};
  \end{tikzpicture}
  \hfill
  \begin{tikzpicture}[y=5mm,every node/.style={font=\scriptsize}]
    \node (x0) at (0,-1) {};
    \node (x1) at (1,-2) {};
    \node (x2) at (2,-1) {};
    \node (x3) at (3,0) {};
    \node (x4) at (4,3) {};

    \fill[blue!20] (x0.center) -- (x1.center) -- (x3.center) -- (x4.center) -- ++(0,1) -- ++(-4,0) -- cycle;
    \draw[red,very thick] (x0.center) -- node[red,above] {$-1$} (x1.center) -- node[red,above] {$1$} (x2.center) -- node[red,above] {$1$} (x3.center) -- node[red,above] {$3$} (x4.center);
    \fill (x0) circle (2.5pt);
    \fill (x1) circle (2.5pt);
    \fill (x2) circle (2.5pt);
    \fill (x3) circle (2.5pt);
    \fill (x4) circle (2.5pt);

    \node[anchor=north] at (x0) {$-1$};
    \node[anchor=north] at (x1) {$-2$};
    \node[anchor=north] at (x2) {$-1$};
    \node[anchor=north] at (x3) {$0$};
    \node[anchor=north] at (x4) {$3$};

    \node[font=\footnotesize] at (2,2.75) {$\text{Res}_{x_3}(f_3,x_1-f_1')$};
  \end{tikzpicture}
  \hfill
  \begin{tikzpicture}[y=3mm,every node/.style={font=\scriptsize}]
    \node (x0) at (0,11) {};
    \node (x1) at (1,7) {};
    \node (x2) at (2,5) {};
    \node (x3) at (3,3) {};
    \node (x4) at (4,3) {};

    \fill[blue!20] (x0.center) -- (x1.center) -- (x3.center) -- (x4.center) -- ++(0,10) -- ++(-4,0) -- cycle;
    \draw[red,very thick] (x0.center) -- node[red,above] {$-4$} (x1.center) -- node[red,above] {$-2$} (x2.center) -- node[red,above] {$-2$} (x3.center) -- node[red,above] {$0$} (x4.center);
    \fill (x0) circle (2.5pt);
    \fill (x1) circle (2.5pt);
    \fill (x2) circle (2.5pt);
    \fill (x3) circle (2.5pt);
    \fill (x4) circle (2.5pt);

    \node[anchor=north] at (x0) {$11$};
    \node[anchor=north] at (x1) {$7$};
    \node[anchor=north] at (x2) {$5$};
    \node[anchor=north] at (x3) {$3$};
    \node[anchor=north] at (x4) {$3$};

    \node[font=\footnotesize] at (2,11) {$\text{Res}_{x_3}(f_3,x_1-f_1'')$};
  \end{tikzpicture}
  \hfill
  \caption{Newton polygons of $f_3$ and the two resultants in Example~\ref{ex:mainAlgorithm}. Below each vertex is its height, above each edge is its slope.}\label{fig:newtonPolygon}
\end{figure}

\section{Implementation}\label{sec:implementation}

In this section, we reflect on some design decisions that were made in the implementation of the algorithms in the \textsc{Singular} library \texttt{tropicalProjection.lib}. While the reader who is only interested in the algorithms, their performance, and their complexity may skip this section without impeding their understanding, we thought it important to include this section for the reader who is interested in the actual implementation.

\subsection{Picking unimodular transformations in Algorithm~\ref{algo:glue} Line~\ref{algstep:pickU}} \label{ssec:implGlue}
As $\pi_u|_{T}$ is injective for generic $u\in\mathcal U$, it seems reasonable to sample random $u \in \mathcal U$ until the corresponding projection is injective on the candidate set. Our implementation however iterates over all $u\in\mathcal U$ in increasing $\ell_1$-norm until the smallest one with injective $\pi_u|_{T}$ is found. This is made in an effort to keep the slim unimodular transformation $\varphi_u(I)$ as simple as possible, since Lines~\ref{algstep:callAlgoUnimodular}--\ref{algstep:computeMinPoly} are the main bottlenecks of our algorithm.

\subsection{Transforming Gr\"obner bases in Algorithm~\ref{algo:glue} Line~\ref{algstep:callAlgoUnimodular}} \label{ssec:implGlue}
As mentioned before, Lines~\ref{algstep:callAlgoUnimodular}--\ref{algstep:computeMinPoly} are the main bottlenecks of our algorithm. Two common reasons why polynomial computations may scale badly are an explosion in degree or in coefficient size. The degree of the polynomials is not problematic in our algorithm, as using Algorithm~\ref{algo:groebnerTransformation} in Line~\ref{algstep:callAlgoUnimodular} only incurs basic arithmetic operations in $K[x_n]/(f_n)$ whose elements can be represented by polynomials of degree bounded by $\deg(f_n)$, while the degree of the minimal polynomial in Line~\ref{algstep:computeMinPoly} also is bounded by $\deg(f_n)$. Therefore, the only aspect that needs to be controlled in our computation is the size of the coefficients.

Coefficient explosion is a common problem for computing inverses in $K[x_n]/(f_n)$ via the Extended Euclidean Algorithm \cite[\S 6.1]{GG13}. To make matters worse, the polynomial $\overbar{h} := \overbar{x_n}^{u_n}\cdot\prod_{i \neq \ell,n} \overbar{f_i}^{u_i} \in K[x_n]/(f_n)$ to be inverted in \cref{algo:groebnerTransformation} usually already has large coefficients. However, we can exploit the fact that the minimal polynomial of $\overbar{f_\ell'} \in K[x_n]/(f_n)$ is $\mu = \sum_{i=0}^k a_i z^i$ if and only if the minimal polynomial of $(\overbar{f_\ell'})^{-1}$ is $\sum_{i=0}^k a_{k-i} z^i$. Instead of computing $\overbar{f_\ell'} = \overbar{h}^{-1} \overbar{f_\ell}$ in \cref{algo:groebnerTransformation}, it therefore suffices to compute $(\overbar{f_\ell'})^{-1} = \overbar{h} \cdot (\overbar{f_\ell})^{-1} \in K[x_n]/(f_n)$, which is easier as $\overbar{f_\ell}$ has generally smaller coefficients than $\overbar{h}$ and is independent of $u$, so its inversion modulo $f_n$ is much faster.

\subsection{Computing minimal polynomials in Algorithm~\ref{algo:glue} Line~\ref{algstep:computeMinPoly}} \label{ssec:implGlue}
The computation of minimal polynomials for elements in $K[x_n]/(f_n)$ can be carried out in many different ways, for example using:
\setlist[description]{font=\normalfont\itshape\textcolor{blue}}
\begin{description}[leftmargin=*]
\item[Resultants] We can compute the resultant of the two polynomials $f_n$ and %
$hx_n-f_\ell \in K[x_\ell, x_n]$ with respect to the variable $x_n$ by standard resultant algorithms. The minimal polynomial $\mu(x_\ell) \in K[x_\ell]$ is the squarefree part of the resultant.
\item[Linear Algebra] Let $k \in \NN$ be minimal such that in the finite-dimensional $K$-vector space $K[x_n]/(f_n)$ the set of polynomials $\{\overbar{h}^{d-i} \overbar{f_\ell}^i \mid i = 0,\ldots, k\}$ is linearly dependent, where $d := \deg(f_n)$. We can find a linear dependence $\sum_{i=0}^k a_i \overbar{h}^{d-i} \overbar{f_\ell}^i = 0$ and conclude that $\mu = \sum_{i=0}^k a_{k-i} z^i$.
\item[Gröbner bases] Note that $\{f_n,x_\ell-f_\ell'\} \subseteq K[x_\ell,x_n]$ forms a Gröbner basis with respect to the lexicographical ordering with $x_n \prec x_\ell$. We can transform this to a Gröbner basis with respect to the lexicographical ordering with $x_\ell \prec x_n$ using \texttt{FGLM} \cite{FGLM93} and read off the eliminant $\mu(x_\ell)$ as the generator of the elimination ideal $(x_\ell-f_\ell',f_n) \cap K[x_\ell]$.
\end{description}
For polynomials with small coefficients, the implementation using \textsc{Singular}'s resultants seemed the fastest, but \textsc{Singular}'s \texttt{FGLM} seems to be best when dealing with very large coefficients.

For $K=\QQ$ however, we can use a modular approach thanks to the \textsc{Singular} library \texttt{modular.lib} \cite{modularlib}: It computes the minimal polynomial over $\FF_p$ for several primes $p$ using any of the above methods, then lifts the results to $\QQ$. This modular approach avoids problems caused by very large coefficients and works particularly well using the method based on linear algebra from above. We can check if the lifted $\mu$ is correct by testing whether $\mu(\overbar{f_\ell'}) = 0$ in $K[x_n]/(f_n)$.

\subsection{Picking gluing strategies in Algorithm~\ref{algo:main} Line~\ref{algstep:pickingGlue}} \label{ssec:gluingStrategies}
\cref{algo:main} is formulated in a flexible way: Different strategies of realising the choice of coordinate sets $A_1,\dots,A_k$ in line~\ref{algstep:pickingGlue} can adapt to the needs of a specific tropicalization problem. The four gluing strategies that follow seem very natural and are implemented in our \textsc{Singular} library. See \cref{fig:strategies} for an illustration in the case $n=5$.
\setlist[description]{font=\normalfont}
\begin{description}[leftmargin=*]
  \item[\texttt{oneProjection}] Only a single iteration of the \texttt{while} loop, in which we pick $k=n$ and $A_i = \{i\}$ for $i=1,\ldots,n$.
  \item[\texttt{sequential}] $n-1$ iterations of the \texttt{while} loop, during which we pick $k=2$ and $A_1 = \{1,\ldots,i\}$ and $A_2 = \{i+1\}$ in the $i$-th iteration.
  \item[\texttt{regularTree($k$)}] $n-1$ iterations of the \texttt{while} loop, which can be partially run in parallel in $\lceil \log_k n \rceil$ batches. In each batch we merge $k$ of the previous projections.
  \item[\texttt{overlap}] $(n-1)n/2$ iterations of the \texttt{while} loop, which can be partially run in parallel in $n-1$ batches.
  During batch $i$, we pick $k=2$ and $A_1 = \{1,\ldots,i\}$, $A_2 = \{1,\ldots,i-1,j\}$ for $j>i$.
\end{description}

\begin{figure}
  \begin{tabular}{cp{0.5em}c}
  \begin{tikzpicture}[level distance=6em, level/.style={sibling distance=6em}, every node/.style = {align=center, scale=0.5}, scale=0.5]
  \node (root) {$\{1,2,3,4,5\}$} [grow' = up]
  child[grow=down, level distance=3em]{edge from parent[draw=none]}
  child
  {
    node {$\{1\}$}
  }
  child
  {
    node {$\{2\}$}
  }
  child
  {
    node {$\{3\}$}
  }
  child
  {
    node {$\{4\}$}
  }
  child
  {
    node {$\{5\}$}
  }
  child[grow=down, level distance=3em]{edge from parent[draw=none]};
  \end{tikzpicture} & &
  \begin{tikzpicture}[level distance=2.75em,level/.style={sibling distance=10em}, every node/.style = {align=center, scale=0.5}, scale=0.5]
  \node {$\{1,2,3,4,5\}$} [grow' = up]
  child
  {
    node {$\{1,2,3,4\}$}
    child
    {
      node {$\{1,2,3\}$}
      child
      {
        node {$\{1,2\}$}
        child
        {
          node {$\{1\}$}
        }
        child[grow=up]{node{$\{2\}$}}
      }child[grow=up]{child{node{$\{3\}$}}}
    }child[grow=up]{child{child{node{$\{4\}$}}}}
  }child[grow=up]{child{child{child{node{$\{5\}$}}}}};
  \end{tikzpicture}\\
  \texttt{oneProjection} & & \texttt{sequential} \\[2em]
  \begin{tikzpicture}[level 1/.style={sibling distance=20em}, level 2/.style={sibling distance=10em}, level 3/.style={sibling distance=5em}, every node/.style = {align=center, scale=0.5}, scale=0.5]
  \def\ld{5em}
  \node (root) {$\{1,2,3,4,5\}$} [grow' = up, level distance=\ld]
  child
  {
    node {$\{1,2,3,4\}$}
    child
    {
      node {$\{1,2\}$}
      child
      {
        node {$\{1\}$}
      }
      child
      {
        node {$\{2\}$}
      }
    }
    child
    {
      node {$\{3,4\}$}
      child
      {
        node {$\{3\}$}
      }
      child
      {
        node {$\{4\}$}
      }
    }
  }
  child[grow=78, level distance=1/sin(78)*\ld]{child{child{node {$\{5\}$}}}};
  \end{tikzpicture} & &
  \begin{tikzpicture}[level distance=3.8em, level/.style={sibling distance=9em}, every node/.style = {align=center, scale=0.5}, scale=0.5]
  \node (root) {$\{1,2,3,4,5\}$} [grow' = up]
  child
  {
    node (p1234) {$\{1,2,3,4\}$}
    child
    {
      node (p123) {$\{1,2,3\}$}
      child
      {
        node (p12) {$\{1,2\}$}
        child
        {
          node (p1) {$\{1\}$}
        }
        child
        {
          node (p2) {$\{2\}$}
        }
      }
      child
      {
        node (p13) {$\{1,3\}$}
        child[missing]
        child
        {
          node (p3) {$\{3\}$}
        }
      }
    }
    child
    {
      node (p124) {$\{1,2,4\}$}
      child[missing]
      child
      {
        node (p14) {$\{1,4\}$}
        child[missing]
        child
        {
          node (p4) {$\{4\}$}
        }
      }
    }
  }
  child
  {
    node (p1235) {$\{1,2,3,5\}$}
    child[missing]
    child
    {
      node (p125) {$\{1,2,5\}$}
      child[missing]
      child
      {
        node (p15) {$\{1,5\}$}
        child[missing]
        child
        {
          node (p5) {$\{5\}$}
        }
      }
    }
  };
  \draw (p13) -- (p1);
  \draw (p14) -- (p1);
  \draw (p15) -- (p1);
  \draw (p124) -- (p12);
  \draw (p125) -- (p12);
  \draw (p1235) -- (p123);
  \end{tikzpicture}
  \\
  \texttt{regularTree(2)} & & \texttt{overlap}
  \end{tabular}
  \caption{Visualisation of different gluing strategies.}
  \label{fig:strategies}
\end{figure}
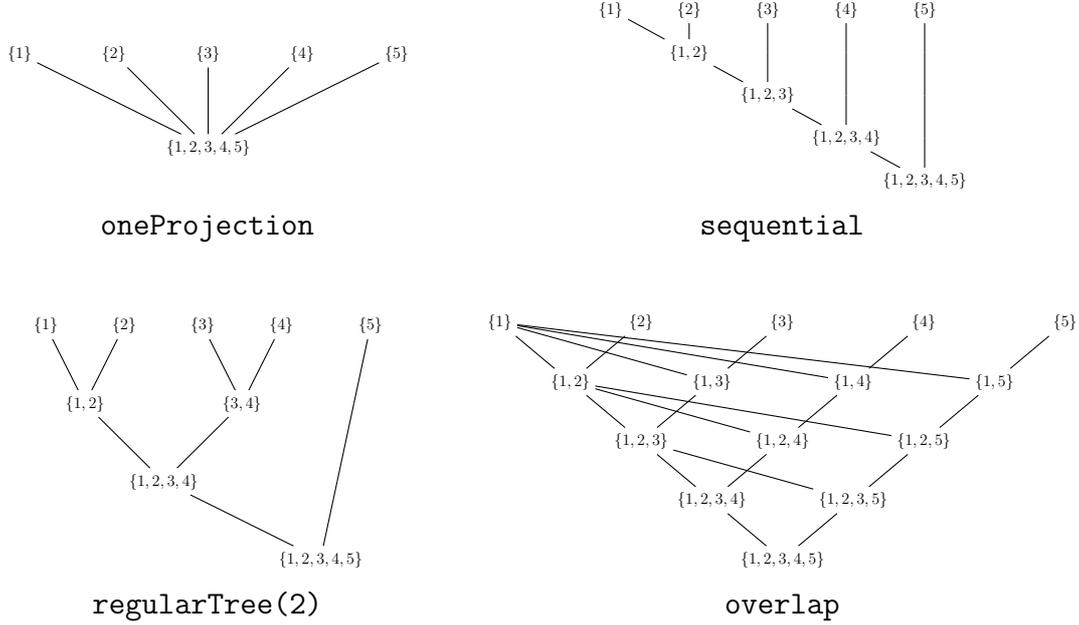

\medskip

\texttt{oneProjection} is the simplest strategy, requiring only one unimodular transformation. For examples of very low degree, it is the best strategy due to its minimal overhead. For examples of higher degree $d$, the candidate set $T$ in \cref{algo:glue} can become quite large, at worst $|T| = d^n$. This generally leads to larger $u \in \mathcal U$ in Line~\ref{algstep:pickU} and causes problems due to coefficient growth.

\texttt{sequential} avoids the problem of a large candidate set $T$ by only gluing two projections at a time, guaranteeing  $|T|\leq d^2$. This comes at the expense of computing $n-1$ unimodular transformations, but even for medium-sized instances we observe considerable improvements compared to \texttt{oneProjection}. In \cref{sec:complexity}, we prove that \texttt{sequential} guarantees good complexity bounds on \cref{algo:main}.

\texttt{regularTree($k$)} can achieve considerable speed-up by parallelisation. Whereas every \texttt{while}-iteration in \texttt{sequential} depends on the output of the previous iteration, \texttt{regularTree($k$)} allows us to compute all gluings in parallel in $\lceil \log_k n\rceil$ batches. The total number of gluings remains the same.

\texttt{overlap} further reduces the size of the candidate set $T$ compared to \texttt{sequential}, while exploiting parallel computation like \texttt{regularTree($k$)}. It glues projections two at a time, but only those $A_1$ and $A_2$ which overlap significantly. This can lead to much smaller candidate sets $T$, at best $|T|=d$ which makes a unimodular transformation obsolete. The strategy \texttt{overlap} seems particularly successful in practice and is the one used for the timings in Section~\ref{sec:timings}.

Our implementation in \textsc{Singular} also allows for custom gluing strategies by means of specifying a graph as in \cref{fig:strategies}.

\section{Timings}\label{sec:timings}

In this section we present timings of our \textsc{Singular} implementation of Algorithm~\ref{algo:main} for $K=\QQ$ and the $2$-adic valuation. We compare it to a \textsc{Magma} \cite{magma} implementation which approximates the roots in the $2$-adic norm. While \textsc{Singular} is also capable of the same task, we chose to compare to \textsc{Magma} instead as it is significantly faster due to its finite precision arithmetic over $p$-adic numbers. Our \textsc{Singular} timings use the \texttt{overlap} strategy, a modular approach and parallelisation with up to four threads. The \textsc{Singular} times we report on are total CPU times across all threads (for reference, the longest example in \textsc{Singular} required 118 seconds total CPU time, but only 32 seconds real time). All computations were run on a server with 2 Intel Xeon Gold 6144 CPUs, 384GB RAM and Debian GNU/Linux 9.9 OS. All examples and scripts are available at \url{https://software.mis.mpg.de}.

\medskip

\subsection{Random lexicographical Gr\"obner bases in shape position}~\label{sec:timingsLexGB}
Given natural numbers $d$ and $n$, a random lexicographical Gr\"obner basis $G$ of an ideal $I\subseteq \QQ[x_1,\dots,x_n]$ of degree $d$ in shape position will be a Gr\"obner basis of the form
\[ G = \{f_n,\, x_{n-1}-f_{n-1},\, x_{n-2}-f_{n-2},\, \ldots,\, x_2-f_2,\, x_1-f_1\}, \]
where $f_n, f_{n-1},f_{n-2} \ldots, f_1$ are univariate polynomials in $x_n$ of degree $d,d-1,d-1,\ldots,d-1$ respectively with coefficients of the form $2^\lambda \cdot (2k+1)$ for a random $\lambda\in\{0,\ldots,99\}$ and a random $k\in\{0,\dots,4999\}$.

Figure \ref{fig:triangularTimings} shows timings for $n=5$ and varying $d$. Each computation was aborted if it failed to terminate within one hour. We see that \textsc{Magma} is significantly faster for small examples, while \textsc{Singular} scales better with increasing degree.

For many of the ideals $I$ however, $\Trop(I)$ has fewer than $d$ distinct points. This puts our algorithm at an advantage, as it allows for easier projections in Algorithm~\ref{algo:main} Line~\ref{algstep:pickU}. Mathematically, it is not an easy task to generate non-trivial examples with distinct tropical points. Picking $f_n$ to have $d$ roots with distinct valuation for example would make all roots live in $\QQ_2$, in which case \textsc{Magma} terminates instantly. Our next special family of examples has criteria which guarantee distinct points.

\newcommand{\tda}{2}
\newcommand{\tdb}{4}
\newcommand{\tdc}{8}
\newcommand{\tdd}{12}
\newcommand{\tde}{16}
\newcommand{\tdf}{20}
\newcommand{\tdg}{24}
\newcommand{\tdh}{2}
\newcommand{\tdi}{2}
\newcommand{\tdj}{2}
\newcommand{\tdk}{2}
\newcommand{\tdl}{2}
\newcommand{\tfa}{100}
\newcommand{\tfb}{100}
\newcommand{\tfc}{100}
\newcommand{\tfd}{100}
\newcommand{\tfe}{100}
\newcommand{\tff}{100}
\newcommand{\tfg}{100}
\newcommand{\tfh}{100}
\newcommand{\tfi}{100}
\newcommand{\tfj}{100}
\newcommand{\tfk}{100}
\newcommand{\tfl}{100}
\newcommand{\tmva}{-1.312}
\newcommand{\tmvb}{-0.213}
\newcommand{\tmvc}{1.610}
\newcommand{\tmvd}{2.822}
\newcommand{\tmve}{3.367}
\newcommand{\tmvf}{3.491}
\newcommand{\tmvg}{3.531}
\newcommand{\tmvh}{2.875}
\newcommand{\tmvi}{2.875}
\newcommand{\tmvj}{2.875}
\newcommand{\tmvk}{2.875}
\newcommand{\tmvl}{2.875}
\newcommand{\tmaa}{-0.959}
\newcommand{\tmab}{1.143}
\newcommand{\tmac}{2.534}
\newcommand{\tmad}{3.556}
\newcommand{\tmae}{3.556}
\newcommand{\tmaf}{3.556}
\newcommand{\tmag}{3.556}
\newcommand{\tmah}{3.164}
\newcommand{\tmai}{3.164}
\newcommand{\tmaj}{3.164}
\newcommand{\tmak}{3.164}
\newcommand{\tmal}{3.164}
\newcommand{\tmia}{-2.000}
\newcommand{\tmib}{-1.523}
\newcommand{\tmic}{-1.097}
\newcommand{\tmid}{0.648}
\newcommand{\tmie}{1.273}
\newcommand{\tmif}{1.937}
\newcommand{\tmig}{0.897}
\newcommand{\tmih}{2.485}
\newcommand{\tmii}{2.485}
\newcommand{\tmij}{2.485}
\newcommand{\tmik}{2.485}
\newcommand{\tmil}{2.485}
\newcommand{\tsva}{0.158}
\newcommand{\tsvb}{0.656}
\newcommand{\tsvc}{1.143}
\newcommand{\tsvd}{1.127}
\newcommand{\tsve}{1.564}
\newcommand{\tsvf}{1.639}
\newcommand{\tsvg}{1.797}
\newcommand{\tsvh}{3.039}
\newcommand{\tsvi}{3.039}
\newcommand{\tsvj}{3.039}
\newcommand{\tsvk}{3.039}
\newcommand{\tsvl}{3.039}
\newcommand{\tsaa}{0.477}
\newcommand{\tsab}{0.954}
\newcommand{\tsac}{1.415}
\newcommand{\tsad}{1.568}
\newcommand{\tsae}{1.778}
\newcommand{\tsaf}{1.833}
\newcommand{\tsag}{2.072}
\newcommand{\tsah}{3.329}
\newcommand{\tsai}{3.329}
\newcommand{\tsaj}{3.329}
\newcommand{\tsak}{3.329}
\newcommand{\tsal}{3.329}
\newcommand{\tsia}{-2.000}
\newcommand{\tsib}{0.301}
\newcommand{\tsic}{0.845}
\newcommand{\tsid}{0.903}
\newcommand{\tsie}{1.431}
\newcommand{\tsif}{1.230}
\newcommand{\tsig}{1.447}
\newcommand{\tsih}{2.415}
\newcommand{\tsii}{2.415}
\newcommand{\tsij}{2.415}
\newcommand{\tsik}{2.415}
\newcommand{\tsil}{2.415}

\begin{figure}[h!]
  \centering
  \begin{tikzpicture}[x=0.5cm,y=1cm,every node/.style={font=\footnotesize}]
    \draw[->] (0,-2.25) -- (27,-2.25) node[anchor=south east] {$\deg(I)$};
    \draw[->] (0,-2.25) -- (0,4) node[anchor=north west] {time (s)};
    \node[red,anchor=west] at (3.5,3.65) {---\textsc{Magma}};
    \node[blue,anchor=west] at (3.5,3.15) {---\textsc{Singular}};
    \coordinate (xva) at (\tda,-2.25);
    \coordinate (xvb) at (\tdb,-2.25);
    \coordinate (xvc) at (\tdc,-2.25);
    \coordinate (xvd) at (\tdd,-2.25);
    \coordinate (xve) at (\tde,-2.25);
    \coordinate (xvf) at (\tdf,-2.25);
    \coordinate (xvg) at (\tdg,-2.25);
    \coordinate (xvh) at (\tdh,-2.25);
    \coordinate (xvi) at (\tdi,-2.25);
    \coordinate (xvj) at (\tdj,-2.25);
    \coordinate (xvk) at (\tdk,-2.25);

    \coordinate (yva) at (0,-2);
    \coordinate (yvb) at (0,-1.5);
    \coordinate (yvc) at (0,-1);
    \coordinate (yvd) at (0,-0.5);
    \coordinate (yve) at (0,0);
    \coordinate (yvf) at (0,0.5);
    \coordinate (yvg) at (0,1);
    \coordinate (yvh) at (0,1.5);
    \coordinate (yvi) at (0,2);
    \coordinate (yvj) at (0,2.5);
    \coordinate (yvk) at (0,3);
    \coordinate (yvl) at (0,3.5);

    \node[font=\tiny] at (xva) {$|$};
    \node[font=\tiny] at (xvb) {$|$};
    \node[font=\tiny] at (xvc) {$|$};
    \node[font=\tiny] at (xvd) {$|$};
    \node[font=\tiny] at (xve) {$|$};
    \node[font=\tiny] at (xvf) {$|$};
    \node[font=\tiny] at (xvg) {$|$};
    \node[font=\scriptsize,anchor=north,yshift=-1mm] at (xva) {2};
    \node[font=\scriptsize,anchor=north,yshift=-1mm] at (xvb) {4};
    \node[font=\scriptsize,anchor=north,yshift=-1mm] at (xvc) {8};
    \node[font=\scriptsize,anchor=north,yshift=-1mm] at (xvd) {12};
    \node[font=\scriptsize,anchor=north,yshift=-1mm] at (xve) {16};
    \node[font=\scriptsize,anchor=north,yshift=-1mm] at (xvf) {20};
    \node[font=\scriptsize,anchor=north,yshift=-1mm] at (xvg) {24};

    \node[xshift=0.175mm] at (yva) {-};
    \node[xshift=0.175mm] at (yvb) {-};
    \node[xshift=0.175mm] at (yvc) {-};
    \node[xshift=0.175mm] at (yvd) {-};
    \node[xshift=0.175mm] at (yve) {-};
    \node[xshift=0.175mm] at (yvf) {-};
    \node[xshift=0.175mm] at (yvg) {-};
    \node[xshift=0.175mm] at (yvh) {-};
    \node[xshift=0.175mm] at (yvi) {-};
    \node[xshift=0.175mm] at (yvj) {-};
    \node[xshift=0.175mm] at (yvk) {-};
    \node[xshift=0.175mm] at (yvl) {-};
    \node[anchor=east,font=\scriptsize] at (yva) {0.01};
    \node[anchor=east,font=\scriptsize] at (yvc) {0.1};
    \node[anchor=east,font=\scriptsize] at (yve) {1};
    \node[anchor=east,font=\scriptsize] at (yvg) {10};
    \node[anchor=east,font=\scriptsize] at (yvi) {100};
    \node[anchor=east,font=\scriptsize] at (yvk) {1000};

    \coordinate[xshift=-0.75mm] (mva) at (\tda,\tmva);
    \coordinate[xshift=-0.75mm] (mvb) at (\tdb,\tmvb);
    \coordinate[xshift=-0.75mm] (mvc) at (\tdc,\tmvc);
    \coordinate[xshift=-0.75mm] (mvd) at (\tdd,\tmvd);
    \coordinate[xshift=-0.75mm] (mve) at (\tde,\tmve);
    \coordinate[xshift=-0.75mm] (mvf) at (\tdf,\tmvf);
    \coordinate[xshift=-0.75mm] (mvg) at (\tdg,\tmvg);
    \coordinate[xshift=-0.75mm] (mvh) at (\tdh,\tmvh);
    \coordinate[xshift=-0.75mm] (mvi) at (\tdi,\tmvi);
    \coordinate[xshift=-0.75mm] (mvj) at (\tdj,\tmvj);
    \coordinate[xshift=-0.75mm] (mvk) at (\tdk,\tmvk);
    \coordinate[xshift=-0.75mm] (maa) at (\tda,\tmaa);
    \coordinate[xshift=-0.75mm] (mab) at (\tdb,\tmab);
    \coordinate[xshift=-0.75mm] (mac) at (\tdc,\tmac);
    \coordinate[xshift=-0.75mm] (mad) at (\tdd,\tmad);
    \coordinate[xshift=-0.75mm] (mae) at (\tde,\tmae);
    \coordinate[xshift=-0.75mm] (maf) at (\tdf,\tmaf);
    \coordinate[xshift=-0.75mm] (mag) at (\tdg,\tmag);
    \coordinate[xshift=-0.75mm] (mah) at (\tdh,\tmah);
    \coordinate[xshift=-0.75mm] (mai) at (\tdi,\tmai);
    \coordinate[xshift=-0.75mm] (maj) at (\tdj,\tmaj);
    \coordinate[xshift=-0.75mm] (mak) at (\tdk,\tmak);
    \coordinate[xshift=-0.75mm] (mia) at (\tda,\tmia);
    \coordinate[xshift=-0.75mm] (mib) at (\tdb,\tmib);
    \coordinate[xshift=-0.75mm] (mic) at (\tdc,\tmic);
    \coordinate[xshift=-0.75mm] (mid) at (\tdd,\tmid);
    \coordinate[xshift=-0.75mm] (mie) at (\tde,\tmie);
    \coordinate[xshift=-0.75mm] (mif) at (\tdf,\tmif);
    \coordinate[xshift=-0.75mm] (mig) at (\tdg,\tmig);
    \coordinate[xshift=-0.75mm] (mih) at (\tdh,\tmih);
    \coordinate[xshift=-0.75mm] (mii) at (\tdi,\tmii);
    \coordinate[xshift=-0.75mm] (mij) at (\tdj,\tmij);
    \coordinate[xshift=-0.75mm] (mik) at (\tdk,\tmik);

    \coordinate[xshift=0.75mm] (sva) at (\tda,\tsva);
    \coordinate[xshift=0.75mm] (svb) at (\tdb,\tsvb);
    \coordinate[xshift=0.75mm] (svc) at (\tdc,\tsvc);
    \coordinate[xshift=0.75mm] (svd) at (\tdd,\tsvd);
    \coordinate[xshift=0.75mm] (sve) at (\tde,\tsve);
    \coordinate[xshift=0.75mm] (svf) at (\tdf,\tsvf);
    \coordinate[xshift=0.75mm] (svg) at (\tdg,\tsvg);
    \coordinate[xshift=0.75mm] (svh) at (\tdh,\tsvh);
    \coordinate[xshift=0.75mm] (svi) at (\tdi,\tsvi);
    \coordinate[xshift=0.75mm] (svj) at (\tdj,\tsvj);
    \coordinate[xshift=0.75mm] (svk) at (\tdk,\tsvk);
    \coordinate[xshift=0.75mm] (saa) at (\tda,\tsaa);
    \coordinate[xshift=0.75mm] (sab) at (\tdb,\tsab);
    \coordinate[xshift=0.75mm] (sac) at (\tdc,\tsac);
    \coordinate[xshift=0.75mm] (sad) at (\tdd,\tsad);
    \coordinate[xshift=0.75mm] (sae) at (\tde,\tsae);
    \coordinate[xshift=0.75mm] (saf) at (\tdf,\tsaf);
    \coordinate[xshift=0.75mm] (sag) at (\tdg,\tsag);
    \coordinate[xshift=0.75mm] (sah) at (\tdh,\tsah);
    \coordinate[xshift=0.75mm] (sai) at (\tdi,\tsai);
    \coordinate[xshift=0.75mm] (saj) at (\tdj,\tsaj);
    \coordinate[xshift=0.75mm] (sak) at (\tdk,\tsak);
    \coordinate[xshift=0.75mm] (sia) at (\tda,\tsia);
    \coordinate[xshift=0.75mm] (sib) at (\tdb,\tsib);
    \coordinate[xshift=0.75mm] (sic) at (\tdc,\tsic);
    \coordinate[xshift=0.75mm] (sid) at (\tdd,\tsid);
    \coordinate[xshift=0.75mm] (sie) at (\tde,\tsie);
    \coordinate[xshift=0.75mm] (sif) at (\tdf,\tsif);
    \coordinate[xshift=0.75mm] (sig) at (\tdg,\tsig);
    \coordinate[xshift=0.75mm] (sih) at (\tdh,\tsih);
    \coordinate[xshift=0.75mm] (sii) at (\tdi,\tsii);
    \coordinate[xshift=0.75mm] (sij) at (\tdj,\tsij);
    \coordinate[xshift=0.75mm] (sik) at (\tdk,\tsik);

    \node[xshift=0.175mm,red] at (maa) {-};
    \node[xshift=0.175mm,red] at (mab) {-};
    \node[xshift=0.175mm,red] at (mac) {-};
    \node[xshift=0.175mm,red] at (mia) {-};
    \node[xshift=0.175mm,red] at (mib) {-};
    \node[xshift=0.175mm,red] at (mic) {-};
    \node[xshift=0.175mm,red] at (mid) {-};
    \node[xshift=0.175mm,red] at (mie) {-};
    \node[xshift=0.175mm,red] at (mif) {-};
    \node[xshift=0.175mm,red] at (mig) {-};
    \draw[red] (maa) -- (mia);
    \draw[red] (mab) -- (mib);
    \draw[red] (mac) -- (mic);
    \draw[red] (mad) -- (mid);
    \draw[red] (mae) -- (mie);
    \draw[red] (maf) -- (mif);
    \draw[red] (mag) -- (mig);
    \draw[red,densely dashed] (mad) -- ++(0,0.5);
    \draw[red,densely dashed] (mae) -- ++(0,0.5);
    \draw[red,densely dashed] (maf) -- ++(0,0.5);
    \draw[red,densely dashed] (mag) -- ++(0,0.5);

    \node[xshift=0.175mm,blue] at (saa) {-};
    \node[xshift=0.175mm,blue] at (sab) {-};
    \node[xshift=0.175mm,blue] at (sac) {-};
    \node[xshift=0.175mm,blue] at (sad) {-};
    \node[xshift=0.175mm,blue] at (sae) {-};
    \node[xshift=0.175mm,blue] at (saf) {-};
    \node[xshift=0.175mm,blue] at (sag) {-};
    \node[xshift=0.175mm,blue] at (sia) {-};
    \node[xshift=0.175mm,blue] at (sib) {-};
    \node[xshift=0.175mm,blue] at (sic) {-};
    \node[xshift=0.175mm,blue] at (sid) {-};
    \node[xshift=0.175mm,blue] at (sie) {-};
    \node[xshift=0.175mm,blue] at (sif) {-};
    \node[xshift=0.175mm,blue] at (sig) {-};
    \draw[blue] (saa) -- (sia);
    \draw[blue] (sab) -- (sib);
    \draw[blue] (sac) -- (sic);
    \draw[blue] (sad) -- (sid);
    \draw[blue] (sae) -- (sie);
    \draw[blue] (saf) -- (sif);
    \draw[blue] (sag) -- (sig);

    \fill[red] (mva) circle (1.5pt);
    \fill[red] (mvb) circle (1.5pt);
    \fill[red] (mvc) circle (1.5pt);
    \fill[white,draw=red] (mvd) circle (1.5pt);
    \fill[white,draw=red] (mve) circle (1.5pt);
    \fill[white,draw=red] (mvf) circle (1.5pt);
    \fill[white,draw=red] (mvg) circle (1.5pt);

    \fill[blue] (sva) circle (1.5pt);
    \fill[blue] (svb) circle (1.5pt);
    \fill[blue] (svc) circle (1.5pt);
    \fill[blue] (svd) circle (1.5pt);
    \fill[blue] (sve) circle (1.5pt);
    \fill[blue] (svf) circle (1.5pt);
    \fill[blue] (svg) circle (1.5pt);
  \end{tikzpicture}\vspace{2mm}
  \begin{small}
   \begin{tabular}{|l|c|c|c|c|c|c|c|c|c|c|c|c|c|}
    \hline
    $\deg(I)$ & 2 & 4 & 8 & 12 & 16 & 20 & 24 \\\hline
    \#\textsc{Singular} finished & 100 & 100 & 100 & 100 & 100 & 100 & 100 \\\hline
    \#\textsc{Magma} finished & 100 & 100 & 100 & 93 & 51 & 21 & 9 \\\hline
    \textsc{Singular} avg. (s) & 1 & 5 & 14 & 19 & 37 & 44 & 63 \\\hline
    \textsc{Magma} avg. (s) & 0 & 1 & 41 & $>$663 & $>$2273 & $>$3095 & $>$3395  \\\hline
  \end{tabular}
  \end{small}\vspace{-2mm}
  \caption{Timings for the randomly generated ideals in shape position.}
\label{fig:triangularTimings}
\end{figure}
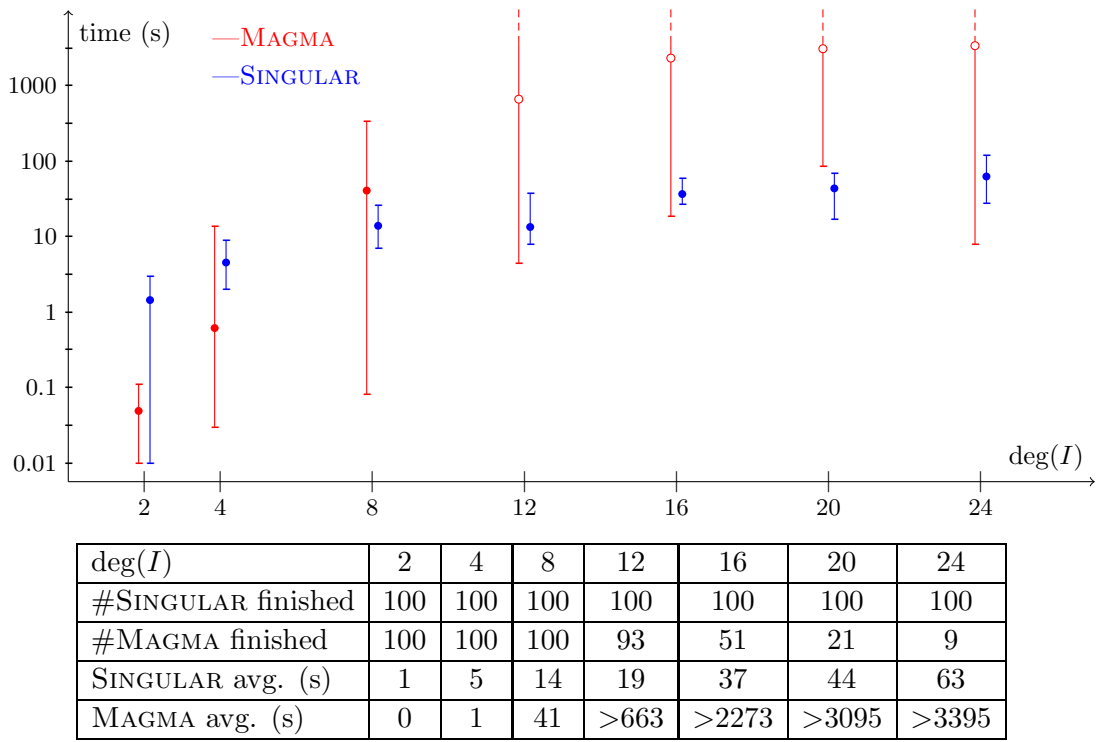

\subsection{Tropical lines on a random honeycomb cubic}~\label{sec:timingsTropicalLines}
Let $V(f)\subseteq\PP^3$ be a smooth cubic surface. In \cite{PV19}, it is shown that $\Trop(f)\subseteq\RR^3$ may contain infinitely many tropical lines. However, for general $f$ whose coefficient valuations induce a honeycomb subdivision of its Newton polytope, $\Trop(f)$ will always contain exactly $27$ distinct tropical lines \cite[Theorem 27]{PV19}, which must therefore be the tropicalizations of the $27$ lines on $V(f)$.

We used \textsc{Polymake} \cite{polymake} to randomly generate $1000$ cubic polynomials with honeycomb subdivisions whose coefficients are pure powers of $2$. For each cubic polynomial $f$, we constructed the one-dimensional homogeneous ideal $\mathcal L_f\subseteq \QQ[p_{12},p_{13},p_{14},p_{23},p_{24},p_{34}]$ of degree~$27$ whose solutions are the lines on $V(f)$ in Pl\"ucker coordinates. Figure~\ref{fig:linesOnCubicsTimings} shows the timings for computing $\Trop(L_f)$, where $L_f:= \mathcal L_f + (p_{34}-1)$ is a zero-dimensional ideal of degree~$27$. Out of our $1000$ random cubics, $8$ had to be discarded because $L_f$ was of lower degree, i.e., $V(f)$ contained lines with $p_{34}=0$.

Unsurprisingly, the \textsc{Singular} timings are relatively stable, while the \textsc{Magma} timings heavily depend on the degree of the splitting field of $L_f$ over $\QQ_2$. Over $\QQ$, the generic splitting field degree would be $51840$ \cite{EJ12}. Over $\QQ_2$, the distinct tropical points of $\Trop(L_f)$ severely restrict the Galois group of the splitting field.

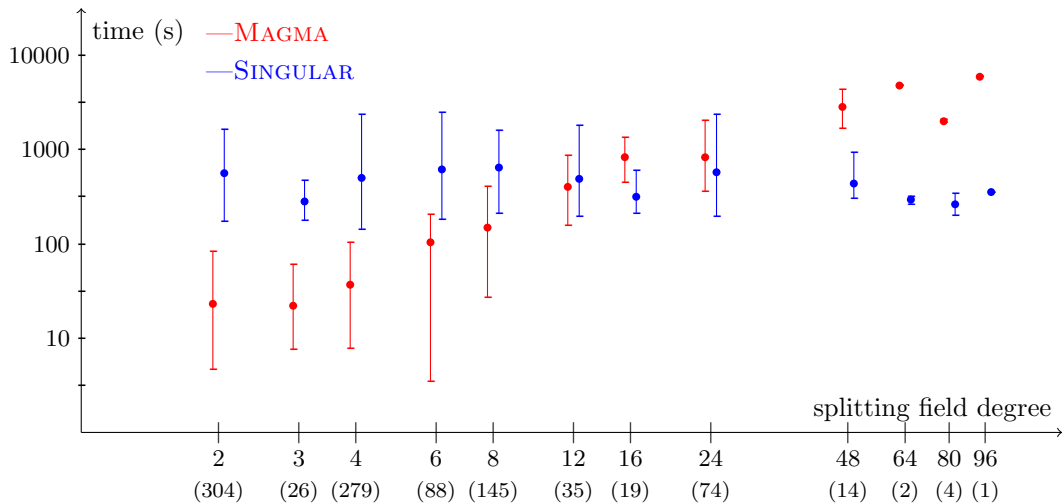
\renewcommand{\tda}{0.301}
\renewcommand{\tdb}{0.477}
\renewcommand{\tdc}{0.602}
\renewcommand{\tdd}{0.778}
\renewcommand{\tde}{0.903}
\renewcommand{\tdf}{1.079}
\renewcommand{\tdg}{1.204}
\renewcommand{\tdh}{1.38}
\renewcommand{\tdi}{1.681}
\renewcommand{\tdj}{1.806}
\renewcommand{\tdk}{1.903}
\renewcommand{\tdl}{1.982}
\renewcommand{\tfa}{304}
\renewcommand{\tfb}{26}
\renewcommand{\tfc}{279}
\renewcommand{\tfd}{88}
\renewcommand{\tfe}{145}
\renewcommand{\tff}{35}
\renewcommand{\tfg}{19}
\renewcommand{\tfh}{74}
\renewcommand{\tfi}{14}
\renewcommand{\tfj}{2}
\renewcommand{\tfk}{4}
\renewcommand{\tfl}{1}
\renewcommand{\tmva}{1.364}
\renewcommand{\tmvb}{1.344}
\renewcommand{\tmvc}{1.567}
\renewcommand{\tmvd}{2.017}
\renewcommand{\tmve}{2.173}
\renewcommand{\tmvf}{2.605}
\renewcommand{\tmvg}{2.92}
\renewcommand{\tmvh}{2.919}
\renewcommand{\tmvi}{3.453}
\renewcommand{\tmvj}{3.68}
\renewcommand{\tmvk}{3.301}
\renewcommand{\tmvl}{3.773}
\renewcommand{\tmaa}{1.923}
\renewcommand{\tmab}{1.783}
\renewcommand{\tmac}{2.012}
\renewcommand{\tmad}{2.309}
\renewcommand{\tmae}{2.608}
\renewcommand{\tmaf}{2.938}
\renewcommand{\tmag}{3.132}
\renewcommand{\tmah}{3.311}
\renewcommand{\tmai}{3.642}
\renewcommand{\tmaj}{3.683}
\renewcommand{\tmak}{3.32}
\renewcommand{\tmal}{3.773}
\renewcommand{\tmia}{0.671}
\renewcommand{\tmib}{0.881}
\renewcommand{\tmic}{0.89}
\renewcommand{\tmid}{0.547}
\renewcommand{\tmie}{1.438}
\renewcommand{\tmif}{2.201}
\renewcommand{\tmig}{2.651}
\renewcommand{\tmih}{2.554}
\renewcommand{\tmii}{3.226}
\renewcommand{\tmij}{3.678}
\renewcommand{\tmik}{3.289}
\renewcommand{\tmil}{3.773}

\renewcommand{\tsva}{2.75}
\renewcommand{\tsvb}{2.45}
\renewcommand{\tsvc}{2.7}
\renewcommand{\tsvd}{2.79}
\renewcommand{\tsve}{2.81}
\renewcommand{\tsvf}{2.69}
\renewcommand{\tsvg}{2.5}
\renewcommand{\tsvh}{2.76}
\renewcommand{\tsvi}{2.64}
\renewcommand{\tsvj}{2.47}
\renewcommand{\tsvk}{2.42}
\renewcommand{\tsvl}{2.55}
\renewcommand{\tsaa}{3.22}
\renewcommand{\tsab}{2.67}
\renewcommand{\tsac}{3.37}
\renewcommand{\tsad}{3.4}
\renewcommand{\tsae}{3.21}
\renewcommand{\tsaf}{3.26}
\renewcommand{\tsag}{2.78}
\renewcommand{\tsah}{3.38}
\renewcommand{\tsai}{2.97}
\renewcommand{\tsaj}{2.51}
\renewcommand{\tsak}{2.54}
\renewcommand{\tsal}{2.55}
\renewcommand{\tsia}{2.24}
\renewcommand{\tsib}{2.25}
\renewcommand{\tsic}{2.16}
\renewcommand{\tsid}{2.26}
\renewcommand{\tsie}{2.32}
\renewcommand{\tsif}{2.29}
\renewcommand{\tsig}{2.33}
\renewcommand{\tsih}{2.29}
\renewcommand{\tsii}{2.48}
\renewcommand{\tsij}{2.42}
\renewcommand{\tsik}{2.3}
\renewcommand{\tsil}{2.55}

\begin{figure}[h!]
  \centering
  \begin{tikzpicture}[x=6cm,y=1.25cm,every node/.style={font=\footnotesize}]
    \draw[->] (0,0) -- (2.15,0) node[anchor=south east] {splitting field degree};
    \draw[->] (0,0) -- (0,4.5) node[anchor=north west] {time (s)};

    \node[red,anchor=west] at (0.25,4.25) {---\textsc{Magma}};
    \node[blue,anchor=west] at (0.25,3.85) {---\textsc{Singular}};

    \coordinate (xva) at (\tda,0);
    \coordinate (xvb) at (\tdb,0);
    \coordinate (xvc) at (\tdc,0);
    \coordinate (xvd) at (\tdd,0);
    \coordinate (xve) at (\tde,0);
    \coordinate (xvf) at (\tdf,0);
    \coordinate (xvg) at (\tdg,0);
    \coordinate (xvh) at (\tdh,0);
    \coordinate (xvi) at (\tdi,0);
    \coordinate (xvj) at (\tdj,0);
    \coordinate (xvk) at (\tdk,0);
    \coordinate (xvl) at (\tdl,0);

    \coordinate (yva) at (0,0.5);
    \coordinate (yvb) at (0,1);
    \coordinate (yvc) at (0,1.5);
    \coordinate (yvd) at (0,2);
    \coordinate (yve) at (0,2.5);
    \coordinate (yvf) at (0,3);
    \coordinate (yvg) at (0,3.5);
    \coordinate (yvh) at (0,4);

    \node[font=\tiny] at (xva) {$|$};
    \node[font=\tiny] at (xvb) {$|$};
    \node[font=\tiny] at (xvc) {$|$};
    \node[font=\tiny] at (xvd) {$|$};
    \node[font=\tiny] at (xve) {$|$};
    \node[font=\tiny] at (xvf) {$|$};
    \node[font=\tiny] at (xvg) {$|$};
    \node[font=\tiny] at (xvh) {$|$};
    \node[font=\tiny] at (xvi) {$|$};
    \node[font=\tiny] at (xvj) {$|$};
    \node[font=\tiny] at (xvk) {$|$};
    \node[font=\tiny] at (xvl) {$|$};
    \node[font=\scriptsize,anchor=north,yshift=-1mm] at (xva) {2};
    \node[font=\scriptsize,anchor=north,yshift=-1mm] at (xvb) {3};
    \node[font=\scriptsize,anchor=north,yshift=-1mm] at (xvc) {4};
    \node[font=\scriptsize,anchor=north,yshift=-1mm] at (xvd) {6};
    \node[font=\scriptsize,anchor=north,yshift=-1mm] at (xve) {8};
    \node[font=\scriptsize,anchor=north,yshift=-1mm] at (xvf) {12};
    \node[font=\scriptsize,anchor=north,yshift=-1mm] at (xvg) {16};
    \node[font=\scriptsize,anchor=north,yshift=-1mm] at (xvh) {24};
    \node[font=\scriptsize,anchor=north,yshift=-1mm] at (xvi) {48};
    \node[font=\scriptsize,anchor=north,yshift=-1mm] at (xvj) {64};
    \node[font=\scriptsize,anchor=north,yshift=-1mm] at (xvk) {80};
    \node[font=\scriptsize,anchor=north,yshift=-1mm] at (xvl) {96};
    \node[font=\tiny,anchor=north,yshift=-5mm] at (xva) {(304)};
    \node[font=\tiny,anchor=north,yshift=-5mm] at (xvb) {(26)};
    \node[font=\tiny,anchor=north,yshift=-5mm] at (xvc) {(279)};
    \node[font=\tiny,anchor=north,yshift=-5mm] at (xvd) {(88)};
    \node[font=\tiny,anchor=north,yshift=-5mm] at (xve) {(145)};
    \node[font=\tiny,anchor=north,yshift=-5mm] at (xvf) {(35)};
    \node[font=\tiny,anchor=north,yshift=-5mm] at (xvg) {(19)};
    \node[font=\tiny,anchor=north,yshift=-5mm] at (xvh) {(74)};
    \node[font=\tiny,anchor=north,yshift=-5mm] at (xvi) {(14)};
    \node[font=\tiny,anchor=north,yshift=-5mm] at (xvj) {(2)};
    \node[font=\tiny,anchor=north,yshift=-5mm] at (xvk) {(4)};
    \node[font=\tiny,anchor=north,yshift=-5mm] at (xvl) {(1)};

    \node[xshift=0.175mm] at (yva) {-};
    \node[xshift=0.175mm] at (yvb) {-};
    \node[xshift=0.175mm] at (yvc) {-};
    \node[xshift=0.175mm] at (yvd) {-};
    \node[xshift=0.175mm] at (yve) {-};
    \node[xshift=0.175mm] at (yvf) {-};
    \node[xshift=0.175mm] at (yvg) {-};
    \node[xshift=0.175mm] at (yvh) {-};
    \node[anchor=east,font=\scriptsize] at (yva) {};
    \node[anchor=east,font=\scriptsize] at (yvb) {10};
    \node[anchor=east,font=\scriptsize] at (yvc) {};
    \node[anchor=east,font=\scriptsize] at (yvd) {100};
    \node[anchor=east,font=\scriptsize] at (yve) {};
    \node[anchor=east,font=\scriptsize] at (yvf) {1000};
    \node[anchor=east,font=\scriptsize] at (yvg) {};
    \node[anchor=east,font=\scriptsize] at (yvh) {10000};

    \coordinate[xshift=-0.75mm] (mva) at (\tda,\tmva);
    \coordinate[xshift=-0.75mm] (mvb) at (\tdb,\tmvb);
    \coordinate[xshift=-0.75mm] (mvc) at (\tdc,\tmvc);
    \coordinate[xshift=-0.75mm] (mvd) at (\tdd,\tmvd);
    \coordinate[xshift=-0.75mm] (mve) at (\tde,\tmve);
    \coordinate[xshift=-0.75mm] (mvf) at (\tdf,\tmvf);
    \coordinate[xshift=-0.75mm] (mvg) at (\tdg,\tmvg);
    \coordinate[xshift=-0.75mm] (mvh) at (\tdh,\tmvh);
    \coordinate[xshift=-0.75mm] (mvi) at (\tdi,\tmvi);
    \coordinate[xshift=-0.75mm] (mvj) at (\tdj,\tmvj);
    \coordinate[xshift=-0.75mm] (mvk) at (\tdk,\tmvk);
    \coordinate[xshift=-0.75mm] (mvl) at (\tdl,\tmvl);
    \coordinate[xshift=-0.75mm] (maa) at (\tda,\tmaa);
    \coordinate[xshift=-0.75mm] (mab) at (\tdb,\tmab);
    \coordinate[xshift=-0.75mm] (mac) at (\tdc,\tmac);
    \coordinate[xshift=-0.75mm] (mad) at (\tdd,\tmad);
    \coordinate[xshift=-0.75mm] (mae) at (\tde,\tmae);
    \coordinate[xshift=-0.75mm] (maf) at (\tdf,\tmaf);
    \coordinate[xshift=-0.75mm] (mag) at (\tdg,\tmag);
    \coordinate[xshift=-0.75mm] (mah) at (\tdh,\tmah);
    \coordinate[xshift=-0.75mm] (mai) at (\tdi,\tmai);
    \coordinate[xshift=-0.75mm] (maj) at (\tdj,\tmaj);
    \coordinate[xshift=-0.75mm] (mak) at (\tdk,\tmak);
    \coordinate[xshift=-0.75mm] (mal) at (\tdl,\tmal);
    \coordinate[xshift=-0.75mm] (mia) at (\tda,\tmia);
    \coordinate[xshift=-0.75mm] (mib) at (\tdb,\tmib);
    \coordinate[xshift=-0.75mm] (mic) at (\tdc,\tmic);
    \coordinate[xshift=-0.75mm] (mid) at (\tdd,\tmid);
    \coordinate[xshift=-0.75mm] (mie) at (\tde,\tmie);
    \coordinate[xshift=-0.75mm] (mif) at (\tdf,\tmif);
    \coordinate[xshift=-0.75mm] (mig) at (\tdg,\tmig);
    \coordinate[xshift=-0.75mm] (mih) at (\tdh,\tmih);
    \coordinate[xshift=-0.75mm] (mii) at (\tdi,\tmii);
    \coordinate[xshift=-0.75mm] (mij) at (\tdj,\tmij);
    \coordinate[xshift=-0.75mm] (mik) at (\tdk,\tmik);
    \coordinate[xshift=-0.75mm] (mil) at (\tdl,\tmil);

    \coordinate[xshift=0.75mm] (sva) at (\tda,\tsva);
    \coordinate[xshift=0.75mm] (svb) at (\tdb,\tsvb);
    \coordinate[xshift=0.75mm] (svc) at (\tdc,\tsvc);
    \coordinate[xshift=0.75mm] (svd) at (\tdd,\tsvd);
    \coordinate[xshift=0.75mm] (sve) at (\tde,\tsve);
    \coordinate[xshift=0.75mm] (svf) at (\tdf,\tsvf);
    \coordinate[xshift=0.75mm] (svg) at (\tdg,\tsvg);
    \coordinate[xshift=0.75mm] (svh) at (\tdh,\tsvh);
    \coordinate[xshift=0.75mm] (svi) at (\tdi,\tsvi);
    \coordinate[xshift=0.75mm] (svj) at (\tdj,\tsvj);
    \coordinate[xshift=0.75mm] (svk) at (\tdk,\tsvk);
    \coordinate[xshift=0.75mm] (svl) at (\tdl,\tsvl);
    \coordinate[xshift=0.75mm] (saa) at (\tda,\tsaa);
    \coordinate[xshift=0.75mm] (sab) at (\tdb,\tsab);
    \coordinate[xshift=0.75mm] (sac) at (\tdc,\tsac);
    \coordinate[xshift=0.75mm] (sad) at (\tdd,\tsad);
    \coordinate[xshift=0.75mm] (sae) at (\tde,\tsae);
    \coordinate[xshift=0.75mm] (saf) at (\tdf,\tsaf);
    \coordinate[xshift=0.75mm] (sag) at (\tdg,\tsag);
    \coordinate[xshift=0.75mm] (sah) at (\tdh,\tsah);
    \coordinate[xshift=0.75mm] (sai) at (\tdi,\tsai);
    \coordinate[xshift=0.75mm] (saj) at (\tdj,\tsaj);
    \coordinate[xshift=0.75mm] (sak) at (\tdk,\tsak);
    \coordinate[xshift=0.75mm] (sal) at (\tdl,\tsal);
    \coordinate[xshift=0.75mm] (sia) at (\tda,\tsia);
    \coordinate[xshift=0.75mm] (sib) at (\tdb,\tsib);
    \coordinate[xshift=0.75mm] (sic) at (\tdc,\tsic);
    \coordinate[xshift=0.75mm] (sid) at (\tdd,\tsid);
    \coordinate[xshift=0.75mm] (sie) at (\tde,\tsie);
    \coordinate[xshift=0.75mm] (sif) at (\tdf,\tsif);
    \coordinate[xshift=0.75mm] (sig) at (\tdg,\tsig);
    \coordinate[xshift=0.75mm] (sih) at (\tdh,\tsih);
    \coordinate[xshift=0.75mm] (sii) at (\tdi,\tsii);
    \coordinate[xshift=0.75mm] (sij) at (\tdj,\tsij);
    \coordinate[xshift=0.75mm] (sik) at (\tdk,\tsik);
    \coordinate[xshift=0.75mm] (sil) at (\tdl,\tsil);

    \node[xshift=0.175mm,red] at (maa) {-};
    \node[xshift=0.175mm,red] at (mab) {-};
    \node[xshift=0.175mm,red] at (mac) {-};
    \node[xshift=0.175mm,red] at (mad) {-};
    \node[xshift=0.175mm,red] at (mae) {-};
    \node[xshift=0.175mm,red] at (maf) {-};
    \node[xshift=0.175mm,red] at (mag) {-};
    \node[xshift=0.175mm,red] at (mah) {-};
    \node[xshift=0.175mm,red] at (mai) {-};
    \node[xshift=0.175mm,red] at (maj) {-};
    \node[xshift=0.175mm,red] at (mak) {-};
    \node[xshift=0.175mm,red] at (mal) {-};
    \node[xshift=0.175mm,red] at (mia) {-};
    \node[xshift=0.175mm,red] at (mib) {-};
    \node[xshift=0.175mm,red] at (mic) {-};
    \node[xshift=0.175mm,red] at (mid) {-};
    \node[xshift=0.175mm,red] at (mie) {-};
    \node[xshift=0.175mm,red] at (mif) {-};
    \node[xshift=0.175mm,red] at (mig) {-};
    \node[xshift=0.175mm,red] at (mih) {-};
    \node[xshift=0.175mm,red] at (mii) {-};
    \node[xshift=0.175mm,red] at (mij) {-};
    \node[xshift=0.175mm,red] at (mik) {-};
    \node[xshift=0.175mm,red] at (mil) {-};
    \draw[red] (maa) -- (mia);
    \draw[red] (mab) -- (mib);
    \draw[red] (mac) -- (mic);
    \draw[red] (mad) -- (mid);
    \draw[red] (mae) -- (mie);
    \draw[red] (maf) -- (mif);
    \draw[red] (mag) -- (mig);
    \draw[red] (mah) -- (mih);
    \draw[red] (mai) -- (mii);
    \draw[red] (maj) -- (mij);
    \draw[red] (mak) -- (mik);
    \draw[red] (mal) -- (mil);

    \node[xshift=0.175mm,blue] at (saa) {-};
    \node[xshift=0.175mm,blue] at (sab) {-};
    \node[xshift=0.175mm,blue] at (sac) {-};
    \node[xshift=0.175mm,blue] at (sad) {-};
    \node[xshift=0.175mm,blue] at (sae) {-};
    \node[xshift=0.175mm,blue] at (saf) {-};
    \node[xshift=0.175mm,blue] at (sag) {-};
    \node[xshift=0.175mm,blue] at (sah) {-};
    \node[xshift=0.175mm,blue] at (sai) {-};
    \node[xshift=0.175mm,blue] at (saj) {-};
    \node[xshift=0.175mm,blue] at (sak) {-};
    \node[xshift=0.175mm,blue] at (sal) {-};
    \node[xshift=0.175mm,blue] at (sia) {-};
    \node[xshift=0.175mm,blue] at (sib) {-};
    \node[xshift=0.175mm,blue] at (sic) {-};
    \node[xshift=0.175mm,blue] at (sid) {-};
    \node[xshift=0.175mm,blue] at (sie) {-};
    \node[xshift=0.175mm,blue] at (sif) {-};
    \node[xshift=0.175mm,blue] at (sig) {-};
    \node[xshift=0.175mm,blue] at (sih) {-};
    \node[xshift=0.175mm,blue] at (sii) {-};
    \node[xshift=0.175mm,blue] at (sij) {-};
    \node[xshift=0.175mm,blue] at (sik) {-};
    \node[xshift=0.175mm,blue] at (sil) {-};
    \draw[blue] (saa) -- (sia);
    \draw[blue] (sab) -- (sib);
    \draw[blue] (sac) -- (sic);
    \draw[blue] (sad) -- (sid);
    \draw[blue] (sae) -- (sie);
    \draw[blue] (saf) -- (sif);
    \draw[blue] (sag) -- (sig);
    \draw[blue] (sah) -- (sih);
    \draw[blue] (sai) -- (sii);
    \draw[blue] (saj) -- (sij);
    \draw[blue] (sak) -- (sik);
    \draw[blue] (sal) -- (sil);

    \fill[red] (mva) circle (1.5pt);
    \fill[red] (mvb) circle (1.5pt);
    \fill[red] (mvc) circle (1.5pt);
    \fill[red] (mvd) circle (1.5pt);
    \fill[red] (mve) circle (1.5pt);
    \fill[red] (mvf) circle (1.5pt);
    \fill[red] (mvg) circle (1.5pt);
    \fill[red] (mvh) circle (1.5pt);
    \fill[red] (mvi) circle (1.5pt);
    \fill[red] (mvj) circle (1.5pt);
    \fill[red] (mvk) circle (1.5pt);
    \fill[red] (mvl) circle (1.5pt);

    \fill[blue] (sva) circle (1.5pt);
    \fill[blue] (svb) circle (1.5pt);
    \fill[blue] (svc) circle (1.5pt);
    \fill[blue] (svd) circle (1.5pt);
    \fill[blue] (sve) circle (1.5pt);
    \fill[blue] (svf) circle (1.5pt);
    \fill[blue] (svg) circle (1.5pt);
    \fill[blue] (svh) circle (1.5pt);
    \fill[blue] (svi) circle (1.5pt);
    \fill[blue] (svj) circle (1.5pt);
    \fill[blue] (svk) circle (1.5pt);
    \fill[blue] (svl) circle (1.5pt);
  \end{tikzpicture}\vspace{2mm}
  \begin{small}
    \begin{tabular}{|l|c|c|c|c|c|c|c|c|c|c|c|c|}
      \hline
      splitting deg. & 2 & 3 & 4 & 6 & 8 & 12 & 16 & 24 & 48 & 64 & 80 & 96\\\hline
      Frequency & 304 & 26 & 279 & 88 & 145 & 35 & 19 & 74 & 14 & 2 & 4 & 1\\\hline
      \textsc{Singular} avg. & 556 & 281 & 505 & 610 & 651 & 490 & 313 & 580 & 440 & 294 & 261 & 352\\\hline
      \textsc{Magma} avg. & 23 & 22 & 37 & 104 & 149 & 403 & 831 & 830 & 2840 & 4791 & 1998 & 5935\\\hline
    \end{tabular}\vspace{-2mm}
  \end{small}
  \caption{Timings for the 27 tropical lines on a tropical honeycomb cubic.}
  \label{fig:linesOnCubicsTimings}
\end{figure}

\section{Complexity}\label{sec:complexity}

In this section, we bound the complexity for computing a zero-dimensional tropical variety from a given Gröbner basis using Algorithm~\ref{algo:main} with the \texttt{sequential} strategy. We show that the number of required arithmetic operations is polynomial in the degree of the ideal and the ambient dimension.
Based on this, we argue that the complexity of computing a higher-dimensional tropical variety is dominated by the Gr\"obner walk required to traverse the Gr\"obner complex, as the computation of a tropical link is essentially polynomial time in the aforementioned sense.

\begin{convention}
  For the remainder of the section, consider a zero-dimensional ideal $I\subseteq K[x_1^{\pm},\dots,x_n^{\pm}]$ of degree $d$ and assume $\nu(K^*)\subseteq\QQ$, so that $\Trop(I) \subseteq \QQ^n$.
\end{convention}

For the sake of convenience, we recall some results on the complexity of arithmetic operations over algebraic extensions, a well-studied topic in the area of computational algebra.

\begin{proposition}[{\cite[Corollary~4.6 + Section~4.3 +  Exercise~12.10]{GG13}}]\label{re:complexity}
  Let $f,g\in K[z]$ be two univariate polynomials of degree $\leq d$. Then:
  \begin{enumerate}[leftmargin=*]
  \item Addition, multiplication and inversion in $K[z]/(f)$ require $\mathcal O(d^2)$ arithmetic operations in $K$.
  \item Computing the $k$-th power of $\bar{g} \in K[z]/(f)$ requires $\mathcal O(d^2 \log k)$ arithmetic operations in $K$.
  \item Computing the minimal polynomial of $\bar{g} \in K[z]/(f)$ requires $\mathcal O(d^2\: \log d \: \log \log d)$ arithmetic operations in $K$.
  \end{enumerate}
\end{proposition}

\begin{proposition}\label{prop:complexityTransform}
  Algorithm~\ref{algo:groebnerTransformation}, which computes the lexicographical Gröbner basis of $\varphi_u(I)$ for some slim transformation $\varphi_u$, requires $\mathcal O\big(d^2 \sum_{u_i > 0} (1+\log u_i)\big)$ arithmetic operations in $K$.
\end{proposition}
\begin{proof}
  We need to count the number of field operations in which the following polynomial $f_\ell'\in K[x_n]$ can be computed:
  \begin{align*}
    f_\ell' \equiv \Big(x_n^{u_n}\cdot\prod_{\substack{i=1\\ i\neq \ell}}^{n-1} f_i^{u_i}\Big)^{-1} \cdot f_\ell \equiv \Big(x_n^{u_n} \cdot f_\ell^{-1} \cdot \prod_{\substack{i=1\\ i\neq \ell}}^{n-1} f_i^{u_i} \Big)^{-1} \pmod{f_n}.
  \end{align*}
  Denoting $k := |\{i \in \{1,\ldots,n\} \mid u_i \neq 0\}|$, this entails the following in the ring $K[x_n]/(f_n)$:
  \begin{itemize}[leftmargin=*]
  \item $k-1$ exponentiations $x_n^{u_n}$ and $f_i^{u_i}$ for $i \neq \ell,n$.
  \item $1$ inversion for $f_\ell$,
  \item $k-1$ multiplications for the product of $f_{\ell}^{-1}$, $x_n^{u_n}$ and all other $f_i^{u_i}$,
  \item $1$ final inversion.
  \end{itemize}
  An exponentiation to the power $u_i$ requires $\mathcal O(d^2 \log u_i)$ arithmetic operations in $K$, while every other operation requires $\mathcal O(d^2)$ arithmetic operations in $K$ by Proposition~\ref{re:complexity}. In total, the number of required field operations in $K$ is
  \[\mathcal O\big(d^2 \sum_{u_i > 0}(1+\log u_i) + d^2+d^2(k-1)+d^2\big) = \mathcal O\big(d^2 \sum_{u_i > 0} (1+\log u_i)\big).\qedhere\]
\end{proof}

\begin{lemma}\label{lem:boundsProj}
  Let $X,Y \subseteq \QQ$ be finite sets of cardinality $\leq d$. Then there exists a non-negative integer $m \leq \binom{d^2}{2}$ such that $X \times Y \to \QQ$, $(a,b) \mapsto a-mb$ is injective. The smallest such $m$ can be found in $\mathcal O(d^4)$ arithmetic operations in $\QQ$.
\end{lemma}
\begin{proof}
  The map $(a,b)\mapsto a-mb$ will fail to be injective if and only if there exists a pair of points in $X\times Y$ lying on an affine line with slope $m$. Since there are at most $\binom{d^2}{2}$ pairs of points, the statement follows by the pigeonhole principle.

  We can determine all integral slopes attained by a line between any two points of $X \times Y$ with $\mathcal O(\binom{d^2}{2}) = \mathcal O(d^4)$ arithmetic operations in $\QQ$. Picking the smallest natural number not occurring among these slopes gives the desired $m$.
\end{proof}

\begin{proposition}\label{prop:complexityGlue}
  Let $k \in \{2,\ldots,n\}$ and assume that the following is known from a previous call of Algorithm~\ref{algo:glue} within Algorithm~\ref{algo:main} running the \texttt{sequential} strategy:
  \begin{itemize}[leftmargin=*]
    \item $p_{\{1,\dots,k-1\}}(\Trop(I))$ and $p_{\{k\}}(\Trop(I))$,
    \item a slim transformation $\varphi_v$ concentrated at $\ell$ with $v_i = 0$ for $i \geq k$ such that $\pi_v$ is injective on $p_{\{1,\dots,k-1\}}(\Trop(I))$,
    \item the lexicographical Gröbner basis of $\varphi_v(I)$.
  \end{itemize}
  Then Algorithm~\ref{algo:glue} for gluing the two projections into $p_{\{1,\dots,k\}}(\Trop(I))$ requires $\mathcal O(d^2 \log d \log \log d)$ and $\mathcal O(d^4)$ arithmetic operations in $K$ and $\QQ$ respectively.
\end{proposition}

\begin{proof}
  Applying Lemma~\ref{lem:boundsProj} to $X:=\pi_v(p_{\{1,\dots,k-1\}}(\Trop(I)))$ and $Y:=p_{\{k\}}(\Trop(I))$, we can compute a minimal $m \leq \binom{d^2}{2}$ such that $(a,b) \mapsto a-m b$ is injective on $X \times Y$ in $\mathcal O(d^4)$ arithmetic $\QQ$-operations. Setting $w := v + m e_k$, this means that $\pi_w$ is injective on $p_{\{1,\dots,k-1\}}(\Trop(I))\times p_{\{k\}}(\Trop(I))$.

  Since $\varphi_w(I) = \varphi_u(\varphi_v(I))$ for $u := m e_k-e_\ell$ and a lexicographical Gr\"obner basis of $\varphi_v(I)$ is already known, we may compute the lexicographical Gr\"obner basis of $\varphi_w(I)$ by applying Algorithm~\ref{algo:groebnerTransformation} to $u$ and $\varphi_v(I)$. By \cref{prop:complexityTransform}, this requires $\mathcal O(d^2 \log m) =\mathcal O(d^2 \log d)$ arithmetic operations in $K$.

  By Proposition~\ref{re:complexity}, computing the minimal polynomial of $\overbar{f_\ell'}\in K[x_n]/(f_n)$ requires $\mathcal O(d^2 \log d \log \log d)$ arithmetic operations in $K$, so the overall number of arithmetic $K$-operations in Algorithm~\ref{algo:glue} is also $\mathcal O(d^2 \log d \log \log d)$.
\end{proof}

\begin{theorem}\label{thm:complexityMain}
  Algorithm~\ref{algo:main}, which computes the zero-dimensional tropical variety $\Trop(I)$, with the \texttt{sequential} strategy requires $\mathcal O(n\, d^2\, \log d\, \log \log d)$ and $\mathcal O(nd^4)$ arithmetic operations in $K$ and $\QQ$ respectively.
\end{theorem}
\begin{proof}
  Algorithm~\ref{algo:main} using the \texttt{sequential} strategy consists of
  \begin{itemize}[leftmargin=*]
  \item Computing minimal polynomials of $\overbar{f_k}\in K[x_n]/(f_n)$ for $k=1,\ldots,n-1$
  \item Applying Algorithm~\ref{algo:glue} to $p_{\{1,\ldots,k-1\}}(\Trop(I))$ and $p_{\{k\}}(\Trop(I))$ for $k=2,\ldots,n$.
  \end{itemize}
  We may store the information on the unimodular transformation computed in iteration $k-1$ during the computation of $p_{\{1,\ldots,k-1\}}(\Trop(I))$ and this information may be used in the next iteration. Then Propositions~\ref{re:complexity} and~\ref{prop:complexityGlue} allow us to deduce the claimed bounds on arithmetic operations in \cref{algo:main}.
\end{proof}

\begin{remark}[Computing positive-dimensional tropical varieties]
  Currently, \textsc{gfan} and \textsc{Singular} are the only software systems capable of computing general tropical varieties, and both rely on a guided traversal of the Gr\"obner complex as introduced in \cite{BJSST07}. Their frameworks roughly consist of two parts:
  \begin{enumerate}[leftmargin=*]
  \item the Gr\"obner walk to traverse the tropical variety,
  \item the computation of tropical links to guide the Gr\"obner walk.
  \end{enumerate}
  While the computation of tropical links had been a major bottleneck of the original algorithm and in early implementations, experiments suggest that it has since been resolved by new approaches \cite{Chan13,HR18}. However,  the algorithm in \cite[\S 4.2]{Chan13} relies heavily on projections, while \cite[Algorithm~2.10]{HR18} relies on root approximations to an unknown precision, so neither approach has good complexity bounds. In fact, \cite[Timing 3.9]{HR18} shows that the necessary precision can be exponential in the number of variables.

  Algorithm~\ref{algo:main} was designed with \cite[Algorithm 2.10]{HR18} in mind, and with Theorem~\ref{thm:complexityMain} we argue that the complexity of calculating tropical links as in \cite[Algorithm~4.6]{HR18} is dominated by the complexity of the Gr\"obner basis computations required for the Gr\"obner walk. In the following, let $J\unlhd K[x_1^{\pm},\ldots,x_n^{\pm}]$ be a homogeneous ideal of codimension $c$ and degree $d$.

  \begin{enumerate}[leftmargin=*]
  \item The Gr\"obner walk requires Gr\"obner bases of initial ideals $\text{in}_w(J)$ with respect to weight vectors $w\in\Trop(J)$ with $\dim C_w(J) = \dim\Trop(J) - 1$, where $C_w(J)$ denotes the Gr\"obner polyhedron of $J$ around $w$. Note that $\text{in}_w(J)$ is neither monomial since $w\in\Trop(J)$ nor binomial as $\dim C_w(J)< \dim\Trop(J)$. Therefore, this is a general Gr\"obner basis computation which is commonly regarded as double exponential time.
  \item Replacing \cite[Algorithm~2.10]{HR18} in \cite[Algorithm~4.6]{HR18} with our Algorithm~\ref{algo:main} requires Gr\"obner bases of ideals of the form
    \[ \text{in}_w(J)|_{x_1=\ldots=x_{c-1}=1,x_{c}=\lambda}\unlhd K[x_{c+1},\ldots,x_n], \]
    where $w$ is chosen as before and $\lambda\in K$ is chosen to satisfy $\nu(\lambda)=\pm 1$. These ideals are zero-dimensional of degree at most $d$, and it is known that Gr\"obner bases of zero-dimensional ideals can be on average computed in polynomial time in the number of solutions \cite{Lakshman1991,LL91}. Thus the entire computation of tropical links can on average be done in polynomial time.
  \end{enumerate}
\end{remark}

\renewcommand{\emph}[1]{\textit{\textcolor{red}{#1}}}

\renewcommand{\emph}[1]{\textit{#1}}
\renewcommand*{\bibfont}{\small}
\printbibliography

\end{document}